\documentclass[11pt]{amsart}

\usepackage[english]{babel} 
\usepackage[T1]{fontenc}
\usepackage[utf8]{inputenc}

\usepackage[left=2.5cm, right=2.5cm, top=2cm, bottom=2cm]{geometry} 

\usepackage{amsmath,amssymb,mathrsfs} 
\usepackage{graphicx}  
\usepackage{float, caption} 
\usepackage{wrapfig} 

\usepackage{framed}
\usepackage{array, multirow, tabularx} 
\usepackage{tikz} 
\usepackage{pifont} 
\usetikzlibrary{matrix,arrows,decorations.pathmorphing,patterns}
\usepackage{mathtools}
\usepackage{enumerate}

\usepackage[bookmarksnumbered=true,pdfstartview=FitH,pdftex,pdfborder={0 0 0},linkcolor=blue,citecolor=blue,colorlinks=true]{hyperref} 

{\newtheorem{de}{Definition}[section]}

{
\newtheorem{theo}[de]{\textsc{Theorem}}}
{
}

{
\newtheorem{prop}[de]{Proposition}}

{
\newtheorem{lem}[de]{Lemma}}

{
\newtheorem{cor}[de]{Corollary}}

{\theoremstyle{remark}
\newtheorem{remar}[de]{Remark}}

{\theoremstyle{remark}
}

{\theoremstyle{remark}
}

{\theoremstyle{remark}
\newtheorem{ex}[de]{Example}}

{\theoremstyle{remark}
\newtheorem{nota}[de]{Notation}}

{
\newtheorem*{theosn}{\textsc{Theorem}}}

{
}

{
}

{

{}
{
}

\newcommand{\dr}{\ensuremath{\partial}}

\newcommand{\dd}{\ensuremath{\mathrm{d}}}

\newcommand{\C}{\ensuremath{\mathbb{C}}}

\newcommand{\ov}[1]{\ensuremath{\overline{#1}}}
\newcommand{\w}{\ensuremath{\omega}}

\newcommand{\co}{\ensuremath{\mathcal{O}}}

\newcommand{\wt}[1]{\ensuremath{\widetilde{#1}}}

\newcommand{\lra}[1]{\left\{#1\right\}}

\newcommand{\Z}{\mathbb{Z}}

\newcommand{\p}{\mathfrak{p}}

\newcommand{\mc}[1]{\mathcal{#1}}
\newcommand{\ms}[1]{\mathscr{#1}}

\newcommand{\der}{\mathrm{Der}}
\newcommand{\lrp}[1]{\left(#1\right)}

\newcommand{\val}{\mathrm{val}}
\newcommand{\N}{\mathbb{N}}

\newcommand{\resxc}{\mathrm{res}_{\ms{X}/\ms{C}}}

\newcommand{\lrb}[1]{\left\langle#1\right\rangle}
\newcommand{\cx}{c_{\ms{X}/\ms{C}}}
\newcommand{\resc}{\mathrm{res}_\ms{C}}
\newcommand{\Po}{\mathrm{Poin}}
\newcommand{\ix}{\mc{I}_\ms{X}}

\title{On a generalization of Solomon-Terao formula for subspace arrangements}
\author[D.~Pol]{Delphine Pol}
\thanks{Research supported by a Japan Society for the Promotion of Science (JSPS) Postdoctoral Fellowship (Short-term) for North American and European Researchers.}
\address{ \linebreak
Delphine Pol\\
Department of Mathematics, Hokkaido University\\
 Kita 10, Nishi 8, Kita-ku\\
 Sapporo 060-0810\\
 Japan
 }
 \email{\href{pol@math.sci.hokudai.ac.jp}{pol@math.sci.hokudai.ac.jp}}
 \date{\today}


\usepackage{verbatim}

\newcommand{\infe}{<}
\newcommand{\supe}{>}

\begin{document}

\begin{abstract}
We investigate in this paper a generalization of Solomon-Terao formula for central equidimensional subspace arrangements. We introduce generalized Solomon-Terao functions based on the Hilbert-Poincar\'e series of the modules of multi-logarithmic forms and logarithmic multi-residues. We show that as in the case of hyperplane arrangements, these Solomon-Terao functions are polynomial. We then prove that if the Solomon-Terao polynomial of the modules of multi-residues satisfies a certain property, then this polynomial is related to the characteristic polynomial of the subspace arrangement. In particular, we prove that this generalized Solomon-Tearo formula holds for any line arrangement of any codimension. 
\end{abstract}

\maketitle

\section{Introduction}

The characteristic polynomial of a subspace arrangement is an important invariant which carries information on the combinatorics and the topology of the arrangement and its complement. Over a finite field, it counts the number of points in the complement of the arrangement (see \cite{athanasiadis:subspace:finite:field}). For real $c$-arrangements, it is related to the Poincar\'e series and the Euler characteristic of the complement (see \cite[Theorem 7.3.1, Theorem 8.2.1]{bjorner-subspace}). In the case of graphic hyperplane arrangements, it coincides with the chromatic polynomial, which is for example related to the "four colour theorem" (see \cite[2.4]{orlik-terao-hyperplanes}).

\smallskip

The characteristic polynomial also appears as a specialization of other polynomials attached to the arrangement. For example, it is related to Tutte polynomials (see for example \cite{yoshinaga-tan-liu} and the references within). In the case of hyperplane arrangements, the characteristic polynomial is given by  a specialization of the Solomon-Terao polynomial associated with the modules of logarithmic vector fields or logarithmic differential forms (\cite{solomon-terao-characteristic} and \cite{orlik-terao-hyperplanes}). This relation is known as Solomon-Terao formula. An algebra called Solomon-Terao algebra which is related to the Solomon-Terao polynomial is introduced in the recent paper \cite{abe-maeno-murai-numata}.

\medskip

The purpose of this paper is to investigate a generalization of Solomon-Terao formula for equidimensional subspace arrangements, and in particular for line arrangements.

\medskip

Multi-logarithmic differential forms and their residues along equidimensional analytic subspaces are introduced in \cite{alektsikhcrass}  and \cite{aleksurvey}, which generalize the case of hypersurfaces (see \cite{saitolog}). Several properties of logarithmic forms along hypersurfaces can be extended to complete intersections or equidimensional subspaces, in particular concerning freeness (see  \cite{polfreeci}). 

\smallskip

Let $\ms{X}$ be an equidimensional subspace arrangement of codimension $k$ in $\C^\ell$. We set $S=\C[x_1,\ldots,x_\ell]$ and for $q\in\N$ we set $\Omega^q$ the module of differential $q$-forms on $\C^\ell$. Let $\mc{I}_\ms{X}\subseteq S$ be the radical ideal of vanishing functions on $\ms{X}$. If $\ms{C}$ is an homogeneous reduced complete intersection defined by a regular sequence $(h_1,\ldots,h_k)$ such that $\ms{X}\subseteq \ms{C}$, we define the ideal $\mc{I}_\ms{C}=\lrb{h_1,\ldots,h_k}\subseteq S$ and we set $h=h_1\cdots h_k$. For $q\in\N$, we define the module of multi-logarithmic differential $q$-forms as:
$$\Omega^q(\log \ms{X}/\ms{C})=\lra{\w\in\frac{1}{h}\Omega^q\ ;\ \mc{I}_\ms{X}\w\subseteq \frac{1}{h}\mc{I}_\ms{C}\Omega^q\text{ and } \dd(\mc{I}_\ms{X})\wedge \w\subseteq \frac{1}{h}\mc{I}_\ms{C}\Omega^{q+1}}.$$

We denote by $\mc{R}_\ms{X}^q$ the module of logarithmic multi-residues of $\Omega^{q+k}(\log \ms{X}/\ms{C})$, which does not depend on the choice of $\ms{C}$ (see definition~\ref{de:res}).

\smallskip

Generalizing the Solomon-Terao polynomial, we define for any finite sequence of finitely generated graded $S$-modules $(M^q)_{0\leqslant q\leqslant n}$ the $\Psi$-function as:
$$\Psi(M^\bullet,x,t)=\sum_{q=0}^n \Po(M^q,x)(t(1-x)-1)^q,$$
where $\Po(M^q,x)$ denotes the Hilbert-Poincar\'e series of $M^q$ (see definition~\ref{de:hilb:poinc}). In particular, if $\ms{A}$ is an hyperplane arrangement and if $M^\bullet=\Omega^\bullet(\log \ms{A})$, it coincides with \cite[Definition 4.130]{orlik-terao-hyperplanes}. By Solomon-Terao formula, for any hyperplane arrangement $\ms{A}$, we have 
$$\Psi(\Omega^\bullet(\log \ms{A}),1,t)=\chi(\ms{A},t)$$
where $\chi(\ms{A},t)$ denotes the characteristic polynomial of $\ms{A}$ (see definition~\ref{de:charpoly}). In this paper, we will also consider $\Psi(\mc{R}_\ms{X}^\bullet,x,t)$ which is more convenient in the case of subspace arrangements. 

\medskip

Let $L(\ms{X})$ be the intersection lattice of $\ms{X}$ (see notation~\ref{nota:lattice}). For all $Y\in L(\ms{X})$, we denote by $\ms{X}_Y$ the subspace arrangement composed of the components of $\ms{X}$ which contain $Y$. 
The main results of this paper are (see theorem~\ref{theo:subspace:solomon:terao} and corollary~\ref{cor-line-psi}):

\begin{theosn}\ 
If $\ms{X}$ is a reduced equidimensional subspace arrangement in $\C^\ell$ and if for all $Y\in L(\ms{X})\backslash\lra{V}$, the condition $\Psi(\mc{R}_{\ms{X}_Y}^\bullet,1,1)=~1$ is satisfied, then for all $Y\in L(\ms{X})$, we have \begin{equation}
\label{eq:intro:theo}
\chi(\ms{X}_Y,t)=t^\ell-\Psi(\mc{R}_{\ms{X}_Y}^\bullet,1,t).
\end{equation} 

In particular, formula \eqref{eq:intro:theo} holds for any line arrangement.
\end{theosn}

We also give an example of a subspace arrangement of dimension 2 which shows that formula~\eqref{eq:intro:theo} is not necessarily satisfied.

\medskip

In section~\ref{sec:def}, we first prove that for any equidimensional subspace arrangement $\ms{X}$ of codimension~$k$ in $\C^\ell$, there exists a reduced subspace arrangement $\ms{C}$ of codimension $k$ which contains $\ms{X}$ and which is a complete intersection (see proposition~\ref{prop:red:ic}). This property is then used to define the modules of multi-logarithmic differential forms $\Omega^q(\log \ms{X}/\ms{C})$ of $\ms{X}$ with respect to $\ms{C}$ (see definition~\ref{de:formes:loga}), and to show that these modules are graded $S$-modules (see lemma~\ref{lem:graded:modules}). We also recall some properties of the modules of multi-logarithmic forms and multi-residues given in \cite{aleksurvey} and \cite{polthese}. We then define the $\Psi$-function associated with a family of finitely generated graded $S$-modules, and the characteristic polynomial of a subspace arrangement. 

\smallskip

In section~\ref{sec:psi}, we first prove that the $\Psi$-function associated with the modules of multi-logarithmic forms of an equidimensional subspace arrangement is a polynomial in $x,x^{-1},t$ (see proposition~\ref{prop:psi:poly}). In proposition~\ref{prop:psi:wt:om} we decompose  the $\Psi$-function associated with the modules of multi-logarithmic forms with respect to the following exact sequence, which holds for all $q\in\N$:
$$0\to\frac{1}{h}\mc{I}_\ms{C}\Omega^q\to \Omega^q(\log \ms{X}/\ms{C})\to \mc{R}_\ms{X}^{q-k}\to 0.$$

The $\Psi$-function of $\frac{1}{h}\mc{I}_\ms{C}\Omega^\bullet$ can be computed thanks to Koszul complex (see proposition~\ref{prop:psi:wt:om}). We then explicitly compute the $\Psi$-function of the module of multi-residues in the case of complete intersection line arrangements thanks to \cite[Th\'eor\`eme 6.1.29]{polthese} (see proposition~\ref{prop:hom:ic:curves:psi}).

\smallskip

Section~\ref{sec:solomon:terao} is devoted to the main result of this paper. As suggested by the computation made in the case of complete intersection line arrangements, we introduce a function $\wt{\Psi}(\ms{X},x,t)$ which coincides with the $\Psi$-function of the modules of multi-logarithmic forms if the codimension is odd (see notation~\ref{nota:wt:psi}), and which satisfies $\wt{\Psi}(\ms{X},1,t)=t^\ell-\Psi(\mc{R}_\ms{X},1,t)$. This function is used to prove our main theorem~\ref{theo:subspace:solomon:terao}.

We then give several examples in section~\ref{sec:ex}. We first consider the case of an arbitrary line arrangement $\ms{X}$. We deduce from the complete intersection case a generating family of the module of multi-residues of $\ms{X}$, which enables us to show that the $\Psi$-function of the module of multi-residues satisfies the required property, so that the generalized Solomon-Terao formula holds for any line arrangement. For higher dimensional subspaces, the generalized Solomon-Terao formula may not be satisfied. We give in subsection~\ref{dim:sup} an example of surface in $\C^4$ for which this formula is not satisfied. The question which remains open is to give characterizations of equidimensional subspace arrangements satisfying the Solomon-Terao formula.

\subsection*{Acknowledgments.} The author is grateful to Masahiko Yoshinaga for pointing out this question, and for helpful discussions and comments and in particular for suggesting her to consider a kind of generic arrangement in proposition~\ref{prop:red:ic}.

\section{Definitions}
\label{sec:def}

We give in this section the definitions and some properties of multi-logarithmic forms and multi-residues along equidimensional subspace arrangements. These notions generalize the logarithmic differential forms along reduced hypersurfaces introduced by K. Saito in \cite{saitolog}. These generalizations appear in \cite{alektsikhcrass}, \cite{alekres} and \cite{aleksurvey} in the case of germs of analytic subspaces. We will also recall results from \cite{polfreeci} and \cite{polthese}.

\subsection{Multi-logarithmic forms and multi-residues}

\subsubsection{Preliminary result}

Let $\ell\in\N$, $\ell\geqslant 2$. We set $S=\C[x_1,\ldots,x_\ell]$.
\begin{de}
A subspace arrangement in $\C^\ell$ is a finite collection $\ms{X}=\lra{\ms{X}_1,\ldots,\ms{X}_s}$ of affine subspaces of $\C^\ell$. We call $\ms{X}$ central if for all $i\in\lra{1,\ldots,s}$, $0\in\ms{X}_i$. We call $\ms{X}$ equidimensional of codimension $k$ if for all $i\in\lra{1,\ldots,s}$, the codimension of $\ms{X}_i$ in $\C^\ell$ is $k$. 
\end{de}

In this paper, we will only consider central subspace arrangements, so that by \textit{subspace arrangement} we will always mean \textit{central subspace arrangement}. In addition, we will sometimes identify a subspace arrangement $\ms{X}=\lra{\ms{X}_1,\ldots,\ms{X}_s}$ with $\bigcup_{i=1}^s \ms{X}_i\subseteq \C^\ell$. 

\begin{nota}
Let $\ms{X}\subseteq \C^\ell$ be an equidimensional subspace arrangement. We denote by $\mc{I}_\ms{X}\subseteq S$ the ideal of vanishing functions on $\ms{X}$. In particular, $\mc{I}_\ms{X}$ is a radical and homogeneous ideal. We set $\mc{O}_\ms{X}=S/\mc{I}_\ms{X}$. We denote by $k$ the codimension of $\ms{X}$ in $\C^\ell$. 
\end{nota}

We will first need the following proposition:

\begin{prop}
\label{prop:red:ic}
There exists a subspace arrangement $\ms{C}$ such that $\ms{C}$ is defined by an homogeneous regular sequence $(h_1,\ldots,h_k)\subseteq \mc{I}_\ms{X}$ and such that the ideal $\mc{I}_\ms{C}=\lrb{h_1,\ldots,h_k}$ is radical. 
\end{prop} 
\begin{proof}
The proof of this proposition consists on finding a kind of generic hyperplane arrangement up to codimension $k$. We will construct this hyperplane arrangement by induction. 

Let $H_{1,1},\ldots,H_{1,k}$ be hyperplanes of $\C^\ell$ such that $\bigcap_{i=1}^k H_{1,i}=\ms{X}_1$ and such that for all $i\in\lra{1,\ldots,k}$, for all $j\in\lra{2,\ldots,s}$, $\ms{X}_j\not\subseteq H_{1,i}$. We set for $j\in\lra{1,\ldots,k}$, $\ms{A}_{1,j}=\lra{H_{1,1},\ldots,H_{1,j}}$.

\smallskip

Let $(i,j)\in\lra{1,\ldots,s}\times\lra{1,\ldots,k}$. We assume $(i,j)\neq (1,1)$. We set $(i,j)^*=(i,j-1)$ if $j\geqslant 2$ and $(i,j)^*=(i-1,k)$ if $j=1$.

For $(i,j)\in\lra{2,\ldots,s}\times\lra{1,\ldots,k}$ we define the hyperplane arrangements $\ms{A}_{i,j}$ inductively from $\ms{A}_{1,k}$ as follows. 

Let $H_{i,j}$ be a hyperplane of $\C^\ell$ such that:
\begin{itemize}
\item $\ms{X}_i\subseteq H_{i,j}$,
\item for all $n\in\lra{1,\ldots,s}$, if $n\neq i$, $\ms{X}_n\not\subseteq H_{i,j}$,
\item for all $1\leqslant p\leqslant k$, for all $H_1,\ldots,H_p\in\ms{A}_{(i,j)^*}$, we have $H_1\cap\dots\cap H_p\not\subseteq H_{i,j}.$
\end{itemize}

Then we define $\ms{A}_{i,j}=\ms{A}_{(i,j)^*}\cup\lra{H_{i,j}}$.

We thus obtain an hyperplane arrangement $\ms{A}_{s,k}$. For all $(i,j)\in\lra{1,\ldots,s}\times\lra{1,\ldots,k}$, let $\alpha_{i,j}$ be a reduced equation of the hyperplane $H_{i,j}$. For all $j\in\lra{1,\ldots,k}$, we define $h_j=\prod_{i=1}^s \alpha_{i,j}$. Let us denote by $\ms{C}$ the variety defined by $(h_1,\ldots,h_k)$. We have:

$$\ms{C}=\bigcup_{(i_1,\ldots,i_k)\in\lra{1,\ldots,s}^k} \Big( H_{i_1,1}\cap\dots\cap H_{i_k,k}  \Big).$$

Let us prove that $\ms{C}$ is a complete intersection containing $\ms{X}$ such that the ideal $\mc{I}_\ms{C}=\lrb{h_1,\ldots,h_k}$ is radical, which is equivalent to proving the following three properties:
\begin{enumerate}[a)]
\item \label{ci:1}For all $(i_1,\ldots,i_k)\in\lra{1,\ldots,s}^k$, the dimension of $H_{i_1,1}\cap \dots \cap H_{i_k,k}$ is $\ell-k$,
\item \label{ci:2}$\ms{X}\subseteq \ms{C}$,
\item \label{ci:3} For all $(i_1,\ldots,i_k)\in\lra{1,\ldots,s}^k$, for all $(j_1,\ldots,j_k)\in\lra{1,\ldots,s}^k$, if $$H_{i_1,1}\cap \dots \cap H_{i_k,k}=H_{j_1,1}\cap\dots\cap H_{j_k,k}$$ then $(i_1,\ldots,i_k)=(j_1,\ldots,j_k)$. 
\end{enumerate}

Let us prove \eqref{ci:1}. It is sufficient to prove that for all $q\in\lra{1,\ldots,k-1}$, for all $(i_1,\ldots,i_{q+1})\in\lra{1,\ldots,s}^{q+1}$, we have:
$$\dim(H_{i_1,1}\cap\dots\cap H_{i_q,q}\cap H_{i_{q+1},q+1})=\dim(H_{i_1,1}\cap \dots \cap H_{i_q,q})-1.$$

Since all the hyperplanes contain the origin we have:
\begin{equation}
\label{eq:dim:H}
\dim(H_{i_1,1}\cap \dots \cap H_{i_{q+1},q+1})-\dim(H_{i_1,1}\cap \dots\cap H_{i_q,q})\in\lra{-1,0}.
\end{equation}

Let us notice that $H_{i_1,1},\ldots,H_{i_q,q}\in\ms{A}_{(i_{q+1},q+1)^*}$, so that by assumption, $H_{i_1,1}\cap \dots \cap H_{i_q,q}\not\subseteq H_{i_{q+1},q+1}$. Therefore, the two dimensions in~\eqref{eq:dim:H} cannot be equal. Hence the result. 

\smallskip

Let us prove \eqref{ci:2}. Let $i\in\lra{1,\ldots,s}$. By assumption, for all $j\in\lra{1,\ldots,k}$, $\ms{X}_i\subseteq H_{i,j}$. Therefore, $\ms{X}_i\subseteq H_{i,1}\cap\dots\cap H_{i,k}\subseteq \ms{C}$. More precisely, from \eqref{ci:1}, $\dim( H_{i,1}\cap\dots\cap H_{i,k})=\ell-k=\dim(\ms{X}_i)$, so that we have $ H_{i,1}\cap\dots\cap H_{i,k}=\ms{X}_i$ since both are vector subspaces of the same dimension.

\smallskip

It remains to prove \eqref{ci:3}, which will show that the complete intersection $\ms{C}$ is reduced. One can notice that since $\co_\ms{C}$ is a complete intersection ring, there is no embedding prime.   Let us assume that $$H_{i_1,1}\cap\dots\cap H_{i_k,k}=H_{j_1,1}\cap \dots \cap H_{j_k,k}.$$

If $i_k\neq j_k$, then we may assume that $i_k\infe j_k$. Then $H_{i_1,1},\ldots,H_{i_k,k}\in\ms{A}_{(j_k,k)^*}$, so that $H_{i_1,1}\cap\dots \cap H_{i_k,k}\not\subseteq H_{j_k,k}$, which is contradictory.

Thus, $i_k=j_k$. Let us prove that $H_{i_1,1}\cap\dots \cap H_{i_{k-1},k-1}=H_{j_1,1}\cap\dots H_{j_{k-1},k-1}$. We have:
$$\dim\Big( \big(H_{i_1,1}\cap \dots\cap H_{i_{k-1},k-1}\cap H_{j_1}\cap \dots H_{j_{k-1},k-1}\big)\cap H_{i_k,k} \Big)=\ell-k.$$

Thus, $\delta:=\dim\big(H_{i_1,1}\cap \dots\cap H_{i_{k-1},k-1}\cap H_{j_1}\cap \dots \cap H_{j_{k-1},k-1}\big)\in\lra{\ell-k,\ell-k+1}$.

Since $\dim(H_{i_1,1}\cap \dots \cap H_{i_{k-1},k-1})=\ell-k+1$, if $\delta=\ell-k$, there exists $n\in\lra{1,\ldots,k-1}$ such that $\dim(H_{i_1,1}\cap \dots \cap H_{i_{k-1},k-1}\cap H_{j_n,n})=\ell-k=\dim(H_{i_1,1}\cap\dots \cap H_{i_{k-1},k-1}\cap H_{j_n,n}\cap H_{i_k,k})$. However, $H_{i_1,1}, \dots, H_{i_{k-1},k-1}, H_{j_n,n}\in \ms{A}_{(i_k,k)^*}$, so that $H_{i_1,1}\cap \dots \cap H_{i_{k-1},k-1}\cap H_{j_n,n}\not\subseteq H_{i_k,k}$, which is a contraction. Thus, $\delta=\ell-k+1$. Since $\dim(H_{i_1,1}\cap\dots \cap H_{i_{k-1},k-1})=\dim(H_{j_1,1}\cap \dots \cap H_{j_{k-1},k-1})=\ell-k+1$, we have:
$$H_{i_1,1}\cap\dots \cap H_{i_{k-1},k-1}=H_{j_1,1}\cap \dots \cap H_{j_{k-1},k-1}.$$

Then, in a completely similar way, one can prove by induction that $(i_1,\ldots,i_k)=(j_1,\ldots,j_k)$.

Therefore, $\ms{C}$ is a reduced complete intersection subspace arrangement of codimension $k$ containing~$\ms{X}$.
\end{proof} 
 
\begin{remar}
This result is not implied by \cite[Proposition 4.2.1]{polthese}, where we prove only that for any germ of reduced equidimensional subspace $\ms{X}\subseteq (\C^\ell,0)$, there exists a germ of reduced complete intersection $\ms{C}$ containing $\ms{X}$  of the same dimension, but where $\ms{C}$ may not be homogeneous even if $\ms{X}$ is homogeneous. 
\end{remar} 

\subsubsection{Definitions}

We fix a reduced complete intersection $\ms{C}$ satisfying proposition~\ref{prop:red:ic}. We set $h=h_1\cdots h_k$. We then define the module of multi-logarithmic forms as follows:

\begin{de}[see \protect{\cite[Definition 10.1]{aleksurvey}}]
\label{de:formes:loga}
Let $q\in\N$. The module of multi-logarithmic $q$-forms along $\ms{X}$ with respect to $\ms{C}$ is defined by:
$$\Omega^q(\log \ms{X}/\ms{C})=\lra{\w\in\frac{1}{h}\Omega^q\ ;\ \mc{I}_\ms{X}\w\subseteq \frac{1}{h}\mc{I}_\ms{C}\Omega^q \text{ and } \dd(\mc{I}_\ms{X})\wedge \w\subseteq \frac{1}{h}\mc{I}_\ms{C}\Omega^{q+1}}.$$
\end{de}

\begin{remar}
If $\ms{C}$ is a reduced complete intersection, we set $$\Omega^q(\log \ms{C}):=\Omega^q(\log \ms{C}/\ms{C})=\lra{\w\in\frac{1}{h}\Omega^q\ ;\ \dd(\mc{I}_\ms{C})\wedge\w\subseteq \frac{1}{h}\mc{I}_\ms{C}\Omega^q}.$$
\end{remar}

\begin{remar}[see \protect{\cite[\S 10]{aleksurvey}}]
\label{remar:incl:xc:c}
We have the following inclusion:
$$\Omega^q(\log \ms{X}/\ms{C})\subseteq \Omega^q(\log \ms{C}).$$
\end{remar}

Let us denote $\ms{C}=\ms{X}_1\cup\dots\cup \ms{X}_s\cup \ms{Y}_1\cup\dots\cup \ms{Y}_r$, where $\ms{Y}_1,\ldots,\ms{Y}_r$ are the irreducible components of $\ms{C}$ which are not in $\ms{X}$. We set $\co_\ms{C}=S/\mc{I}_\ms{C}$.

For a ring $R$, we denote by $\mathrm{Frac}(R)$ the total ring of fractions of $R$. 
\begin{de}[\protect{\cite[1.3]{kerskenregulare}}]
\label{de:fund:class}
Let $c_{\ms{X}/\ms{C}}\in\Omega^k$. We say that $\cx$ is a fundamental form of $\ms{X}$ if we have:
$$\ov{\cx}=\ov{\beta_{\ms{X}/\ms{C}}\dd h_1\wedge\dots\wedge \dd h_k}\in\frac{\Omega^k}{\mc{I}_\ms{C}\Omega^k}$$
where $\beta_{\ms{X}/\ms{C}}\in\mathrm{Frac}(\co_\ms{C})$ satisfies for all $i\in\lra{1,\ldots,s}$, $\beta_{\ms{X}/\ms{C}}\Big|_{\ms{X}_i}=1$ and for all $j\in\lra{1,\ldots,r}$, $\beta_{\ms{X}/\ms{C}}\Big|_{\ms{Y}_i}=0$, and $\ov{\cx}$ denotes the class of $\cx$ in $\frac{\Omega^k}{\mc{I}_\ms{C}\Omega^k}$. 
\end{de}

\begin{remar}
In particular, if $\ms{X}=\ms{C}$ is a reduced complete intersection, then one can take $c_{\ms{C}/\ms{C}}=\dd h_1\wedge \dots \wedge \dd h_k$. 
\end{remar}

In order to introduce the module of logarithmic multi-residues, we need the following theorem, which gives a characterization of multi-logarithmic forms and generalizes \cite[(1.1)]{saitolog}.

\begin{theo}[see \protect{\cite[\S 3, Theorem 1]{alekres}} and \protect{\cite[Proposition 4.2.6]{polthese}}] 
\label{theo:carac:loga}
Let $q\in\N$. 
Let $\w\in\frac{1}{h}\Omega^q$. Then $\w\in\Omega^q(\log \ms{X}/\ms{C})$ if and only if there exist $g\in S$ inducing a non zero divisor\footnote{one can choose $g$ as a linear combination of the maximal minors of the Jacobian matrix associated with $(h_1,\ldots,h_k)$.} in $\co_\ms{C}$, $\xi\in\Omega^{q-k}$ and $\eta\in\frac{1}{h}\mc{I}_\ms{C}\Omega^q$ such that \begin{equation}
g\w=\frac{\cx}{h}\wedge\xi+\eta
\end{equation}
\end{theo}

\begin{de}
\label{de:res}
Let $q\in\N$. Let $\w\in\Omega^q(\log \ms{X}/\ms{C})$. With the notations of the theorem, the multi-residue of $\w$ is:
$$\resxc(\w)=\frac{\xi}{g}\in\Omega^{q-k}_\ms{X}\otimes \mathrm{Frac}(\co_\ms{X})$$
where for $p\in\N$, $\Omega^p_\ms{X}=\frac{\Omega^p}{\mc{I}_\ms{X}\Omega^p+\dd (\mc{I}_\ms{X})\wedge\Omega^{p-1}}$
\end{de}

\begin{nota}
For $q\in\N$, we denote $\mc{R}_\ms{X}^q=\resxc(\Omega^{q+k}(\log \ms{X}/\ms{C}))$. 
\end{nota}

\begin{remar}
For all $q\in\N$, the module $\mc{R}_\ms{X}^q$ depends only on $\ms{X}$, contrary to the modules $\Omega^q(\log \ms{X}/\ms{C})$ which depend on the choice of the complete intersection $\ms{C}$ (see \cite[\S 10]{aleksurvey} and \cite[Proposition 4.1.5 and 4.1.13]{polthese}).
\end{remar}
\begin{remar}
\label{remar:wq}
For all $q\in\N$, the module of multi-residues $\mc{R}_\ms{X}^q$ is isomorphic to the module of regular meromorphic forms $\w_\ms{X}^q$(see \cite[\S 10]{aleksurvey}). The module $\w_\ms{X}^{\ell-k}$ is the dualizing module given by $\w_\ms{X}^{\ell-k}=\mathrm{Ext}^k_S(\co_\ms{X},\Omega^\ell)$ and for all $q$, $\w_\ms{X}^q=\mathrm{Hom}_{\co_\ms{X}}(\Omega^{\ell-k-q}_\ms{X},\w_{\ms{X}}^{\ell-k})$.
\end{remar}

\begin{prop}[\protect{\cite[Theorem 10.2]{aleksurvey}} and \protect{\cite[Corollaire 4.1.9]{polthese}}]
For all $q\in\N$, we have the following exact sequence:
\begin{equation}
\label{eq:wtilde:log:rx}
0\to \frac{1}{h}\mc{I}_\ms{C}\Omega^q\to \Omega^q(\log \ms{X}/\ms{C})\to \mc{R}_\ms{X}^{q-k}\to 0.
\end{equation}
\end{prop}

\begin{remar}
\label{remar:infe:k}
For all $q\infe k$, $\Omega^q(\log \ms{X}/\ms{C})=\frac{1}{h}\mc{I}_\ms{C}\Omega^q$.
\end{remar}

\begin{remar}
\label{remar:incl:res}
By remark~\ref{remar:incl:xc:c}, we have $\Omega^q(\log \ms{X}/\ms{C})\subseteq \Omega^q(\log \ms{C})$. If $\w\in\Omega^q(\log \ms{X}/\ms{C})$, we also have $\resxc(\w)=\resc(\w)\big|_\ms{X}\in\Omega^{q-k}_\ms{X}\otimes\mathrm{Frac}(\co_\ms{X})$. In addition, we have an inclusion $\mc{R}_{\ms{X}}^q\hookrightarrow \mc{R}_{\ms{C}}^q$ which is defined as follows: for any $\w\in\Omega^q(\log \ms{X}/\ms{C})$, $\resxc(\w)\in \mc{R}_{\ms{X}}^q\mapsto \mathrm{res}_\ms{C}(\w)\in\mc{R}_\ms{C}^q$. Since the kernel of $\resxc$ is $\frac{1}{h}\mc{I}_{\ms{C}}\Omega^q$ which is also the kernel of $\mathrm{res}_{\ms{C}}$, the previous map is well defined and is an inclusion. 
\end{remar}

\subsection{Poincar\'e polynomials, $\Psi$-functions and characteristic polynomial}

\subsubsection{Graduation}

We keep the same notations as before. In order to consider the Hilbert-Poincar\'e series of $\Omega^\bullet(\log \ms{X}/\ms{C})$ and $\mc{R}_\ms{X}^\bullet$, we first need to introduce a graduation on these modules. 

\begin{nota} 
For $q\in\lra{1,\ldots,\ell}$ and $I=\lra{i_1,\ldots,i_q}\subseteq \lra{1,\ldots,\ell}$, we set $\dd x_I=\dd x_{i_1}\wedge \dots\wedge \dd x_{i_q}$. We denote by $|I|$ the cardinality of the set $I$. 
\end{nota}

\begin{de}
A form $\alpha=\sum_{|I|=q} a_I\dd x_I\in \Omega^q$ is called homogeneous of polynomial degree $p$ if for all $I$, $a_I$ is an homogeneous polynomial of degree $p$.
\end{de}

\begin{lem}
\label{lem:graded:modules}
The modules $\Omega^q(\log\ms{X}/\ms{C})$ are finitely generated graded $S$-modules with the graduation induced by  the polynomial degree on $\Omega^q$. 
\end{lem}
\begin{proof}

Let $(f_1,\ldots,f_r)$ be homogeneous polynomials such that $\lrb{f_1,\ldots,f_r}=\ix$. For all $i\in\lra{1,\ldots,r}$, we set $\delta_i=\deg(f_i)$.

 The module $\Omega^q(\log \ms{X}/\ms{C})$ can be seen as the kernel of the morphism of graded modules:
$$\frac{1}{h}\Omega^q\xrightarrow{\phi} \bigoplus_{i=1}^r \frac{\Omega^q}{\mc{I}_\ms{C}\Omega^q}(\delta_i)\oplus \bigoplus_{i=1}^r \frac{\Omega^{q+1}}{\mc{I}_{\ms{C}}\Omega^{q+1}}(\delta_i-1)$$
defined by $\phi(\w)=\lrp{\lrp{f_i h\w}_{i\in\lra{1,\ldots,r}},\lrp{\dd f_i\wedge h\w}_{i\in\lra{1,\ldots,r}}}$. Therefore, $\Omega^q(\log\ms{X}/\ms{C})$ is a graded $S$-module.
\end{proof}

\begin{nota}
For all $i\in\lra{1,\ldots,k}$, we denote by $d_i$ the degree of the homogeneous polynomial $h_i$, and we set $d=\sum_{i=1}^k d_i=\deg(h)$. 

For all $K\subseteq \lra{1,\ldots,\ell}$ with $|K|=k$, we set $J_K$ the $(k\times k)$-minor of the Jacobian matrix of $(h_1,\ldots,h_k)$ relative to the set $K$. In particular, $\dd h_1\wedge \dots\wedge \dd h_k=\sum_{|K|=k} J_K \dd x_K$. In addition, for all $K\subseteq \lra{1,\ldots,\ell}$ with $|K|=k$, we have $J_K=0$ or $\deg(J_K)=d-k$. Thus, $\deg(\dd h_1\wedge \dots\wedge \dd h_k)=d-k$.
\end{nota}

\begin{prop}
\label{prop:resxc:deg}
The map $\resxc$ is homogeneous of degree $k$.
\end{prop}
\begin{proof}
Let $\w\in\Omega^q(\log \ms{X}/\ms{C})\subseteq \Omega^q(\log \ms{C})$ be an homogeneous multi-logarithmic form. 
Since $\w$ is a multi-logarithmic form along $\ms{C}$, there exist $g,\xi,\eta$ as in theorem~\ref{theo:carac:loga} such that $$g\w=\frac{\dd h_1\wedge \dots\wedge \dd h_k}{h}\wedge \xi+\eta.$$
In particular, since $g$ can be chosen as a $\C$-linear combination of the maximal minors of the Jacobian matrix, one can assume that $g$ is homogeneous of degree $d-k$, so that $g\w$ is an homogeneous form. Therefore, we may assume that $\xi$ and $\eta$ are homogeneous elements respectively of degree $\deg(\w)+d$ and $d-k+\deg(\w)$. Then $\resxc(\w)=\frac{\xi}{g}$ is homogeneous of degree $\deg(\w)+k$. Hence the result. 
\end{proof}

\begin{remar}
\label{remar:wq:degree}
As mentioned in remark~\ref{remar:wq}, the modules of multi-residues are isomorphic to the modules of regular meromorphic forms $\w_\ms{X}^q$. 
Using \cite[(1.2)]{kerskenregulare}, one can prove that the isomorphism $\mc{R}_\ms{X}^q\to \w_\ms{X}^q$ is homogeneous of degree $k$. Indeed, by \cite[(1.2)]{kerskenregulare}, we have: 
\begin{align*}\w^q_\ms{X}=&\left\{\begin{bmatrix}
\alpha\\
f_1,\ldots,f_k
\end{bmatrix} ; \alpha\in\Omega^{q+k}, (f_1,\ldots,f_k)\subseteq \mc{I}_\ms{X} \text{ a regular sequence},\right.
\\   & \left.\ix\alpha\subseteq (f_1,\ldots,f_k)\Omega^{k+q} \text{ and } \dd \ix\wedge \alpha\subseteq (f_1,\ldots,f_k)\Omega^{q+k+1}\right\}
\end{align*}
where $\begin{bmatrix}
\alpha\\
f_1,\ldots,f_k
\end{bmatrix}$ denotes residue symbols (see \cite{kerskennenner}). The isomorphism between $\mc{R}_\ms{X}^q$ and $\w_\ms{X}^q$ is then given by:
$\sigma : \resxc(\w)\in\mc{R}_\ms{X}^q\mapsto \begin{bmatrix}
h\w\\
h_1,\ldots,h_k
\end{bmatrix}\in \w_\ms{X}^q$. Due to \cite[\S2]{kerskennenner}, for $(f_1',\ldots,f_k')\subseteq\lrb{f_1,\ldots,f_k}$ a regular sequence, we have  $\begin{bmatrix}
\alpha\\
f_1,\ldots,f_k
\end{bmatrix}=\begin{bmatrix}
\Delta\alpha\\
f_1',\ldots,f_k'
\end{bmatrix}$ where $\Delta$ is the determinant of the transition matrix from $(f_1,\ldots,f_k) $ to $(f_1',\ldots,f_k')$. If $(f_1,\ldots,f_k)$ is an homogeneous regular sequence, it is then convenient to set $\deg\lrp{\begin{bmatrix}
\alpha\\
f_1,\ldots,f_k
\end{bmatrix}}=\deg(\alpha)-\deg(f_1\cdots f_k)$ so that the degree does not depend on the choice of the homogeneous regular sequence $(f_1,\ldots,f_k)$. It is then easy to see that $\sigma$ has degree $k$. 
\end{remar}

\subsubsection{Characteristic polynomial}

We recall here some useful combinatorial invariants associated with a subspace arrangement which can be found for example in \cite{bjorner-subspace}. 

We first recall the definition of the Hilbert-Poincar\'e series associated with a finitely generated graded $S$-module:

\begin{de}
\label{de:hilb:poinc}
Let $M=\bigoplus_{p\geqslant p_0}$ be a finitely generated graded $S$-module such that each $M_p$ is finite dimensional over $\C$. The Hilbert-Poincar\'e series of $M$ is:
$$\Po(M,x)=\sum_{p\geqslant p_0} (\dim_\C M_p) x^p.$$
\end{de}

The following definition generalizes the Solomon-Poincar\'e polynomial  defined in \cite[Defintion 4.130]{orlik-terao-hyperplanes}:
\begin{de}
Let $(M^q)_{0\leqslant q\leqslant n}$ be a finite sequence of finitely generated graded $S$-modules. We define the $\Psi$-function associated with $(M^q)_q$ as:
$$\Psi(M^\bullet,x,t)=\sum_{q=0}^n \Po(M^q,x)(t(1-x)-1)^q.$$ 
\end{de}

\begin{nota}
\label{nota:lattice}
The intersection lattice $L(\ms{X})$ of $\ms{X}$ is defined by:
$$L(\ms{X})=\lra{\bigcap_{i\in I} \ms{X}_{i}\ ;\ I\subseteq \lra{1,\ldots,s}}.$$

In particular, if $I=\emptyset$, we set $\bigcap_{i\in I} \ms{X}_i=\C^\ell$. We consider the partial order on $L(\ms{X})$ given by reverse inclusion, so that for all $X,Y\in L(\ms{X})$, $X\leqslant Y$ means $Y\subseteq X$. 

\end{nota}
\begin{de}
The M\"obius function is defined on $L(\ms{X})\times L(\ms{X})$ by:
\begin{enumerate}
\item for all $X\in L(\ms{X})$, $\mu(X,X)=1$,
\item for all $X,Y\in L(\ms{X})$ such that $X\infe Y$, $\sum_{\substack{Z\in L(\ms{X})\\ X\leqslant Z\leqslant Y}} \mu(X,Z)=0$,
\item $\mu(X,Y)=0$ otherwise.
\end{enumerate}
For all $X\in L(\ms{X})$, we set $\mu(X)=\mu(V,X)$. 
\end{de}

\begin{de}[\protect{\cite[(4.4.1)]{bjorner-subspace}}]
\label{de:charpoly}
The characteristic polynomial of $\ms{X}$ is defined by:
$$\chi(\ms{X},t)=\sum_{X\in L(\ms{X})} \mu(X)t^{\dim(X)}.$$
\end{de}

The purpose of this paper is to study a generalization of the following theorem:

\begin{theo}[Solomon Terao formula, \protect{\cite[Theorem 4.136]{orlik-terao-hyperplanes}}]
\label{theo:solomon:terao}
Let $\ms{A}$ be an hyperplane arrangement in $\C^\ell$. Then:
$$\chi(\ms{A},t)=\Psi(\Omega^\bullet(\log \ms{A}),1,t).$$
\end{theo}

\section{Properties of the $\Psi$-function}
\label{sec:psi}

We investigate some properties of the $\Psi$-functions associated with the modules of multi-logarithmic forms $\Omega^\bullet(\log \ms{X}/\ms{C})$, the modules $\frac{1}{h}\mc{I}_\ms{C}\Omega^\bullet$ and the multi-residues $\mc{R}_\ms{X}^\bullet$ for an equidimensional subspace arrangement $\ms{X}$ contained in a reduced complete intersection subspace arrangement~$\ms{C}$. 

\subsection{Polynomial}

Let us prove the following property, which generalizes~\cite[Proposition 4.133]{orlik-terao-hyperplanes}.

\begin{prop}
\label{prop:psi:poly}
The $\Psi$-function associated with the modules of multi-logarithmic forms $\Omega^\bullet(\log \ms{X}/\ms{C})$ satisfies:
$$\Psi\lrp{\Omega^\bullet(\log \ms{X}/\ms{C}), x, t}\in\Z[x,x^{-1},t].$$
\end{prop}
\begin{proof}
The proof we suggest here is a generalization of the proof of \cite[Proposition 4.133]{orlik-terao-hyperplanes} in the context of subspace arrangements. Given a vector space $W$ we will denote by $\Omega^q[W]$ the module of $q$-differential forms on $W$. For $p\in\N$, we denote by $\Omega^q[W]_p$ the module of homogeneous $q$-forms of degree $p$ on $W$. We set $V=\C^\ell$. In particular, we have $\Omega^q=\Omega^q[V]$. 

\begin{lem}
Let $\eta\in \Omega^1$. Then for all $\w\in\Omega^q(\log \ms{X}/\ms{C})$, $\eta\wedge\w\in\Omega^{q+1}(\log \ms{X}/\ms{C})$. 
\end{lem}
\begin{proof}
It comes from the fact that $\eta\wedge\frac{1}{h}\mc{I}_\ms{C}\Omega^q\subseteq \frac{1}{h}\mc{I}_\ms{C}\Omega^{q+1}$.  
\end{proof}

\begin{de}
 Let $\eta\in\Omega^1$ be an homogeneous $1$-form of degree $p$. For all $q\in\lra{0,\ldots,\ell-1}$ we define  $\dr_\eta : \Omega^q(\log \ms{X}/\ms{C})\to \Omega^{q+1}(\log \ms{X}/\ms{C})$ by $\dr_\eta(\w)=\eta\wedge\w$. This map is homogeneous of degree~$p$. The $\eta$-complex is the complex:
$$0\to \Omega^0(\log \ms{X}/\ms{C})\xrightarrow{\dr_\eta}\Omega^1(\log \ms{X}/\ms{C})\xrightarrow{\dr_\eta} \dots\xrightarrow{\dr_\eta} \Omega^\ell(\log \ms{X}/\ms{C})\to 0.$$
\end{de}

Our purpose is to prove that for a generic $\eta$, the cohomology groups of the $\eta$-complex are finite dimensional over $\C$ as in \cite[Proposition 4.91]{orlik-terao-hyperplanes}.

We consider the Zariski topology on the vector subspace $S_p$ of $S$ composed of the homogeneous polynomials of degree $p$.

We recall that by \cite[Lemma 4.88]{orlik-terao-hyperplanes}, if $W$ is an $m$-dimensional $\C$-vector space, a generic $\w\in\Omega^1[W]_p$ vanishes only at the origin.

Let $X\in L(\ms{X})$ with $\dim(X)>0$. As in \cite[Lemma 4.89]{orlik-terao-hyperplanes}, the restriction map $$r_{V,X} : \Omega^1_{p}\to \Omega^1[X]_p$$ is continuous with respect to the Zariski topology. We set $$N_p^X=\lra{\w\in\Omega^1_p\ ;\ r_{V,X}(\w) \text{ vanishes only at the origin of }X}.$$ The set $N_p^X$ is an open dense subset in $\Omega^1_p$ since $r_{V,X}$ is continuous. 

\begin{nota}
\label{nota:nd}
We set:
$$N_p=\bigcap_{\substack{X\in L(\ms{X})\\ \dim(X)>0}} N_p^X$$
\end{nota}
Since $L(\ms{X})$ is finite, $N_p$ is an open dense subset of $\Omega^1_p$. 

\begin{de}
The module of logarithmic vector fields along $\ms{X}$ is defined by $$\der(-\log\ms{X})=\lra{\delta\in\Theta\ ;\ \delta(\mc{I}_\ms{X})\subseteq \mc{I}_\ms{X}}.$$
\end{de}

\begin{lem}
\label{lem:eta:maximal:ideal}
Let $\eta\in N_p$. Then the radical of the ideal $$ I(\eta)=\lra{\lrb{\eta,\delta},\delta\in\der(-\log \ms{X})}$$
contains the maximal ideal $S_+=\bigoplus_{q>0} S_q$. 
\end{lem}

\begin{proof}
It suffices to prove that the zero locus of $I(\eta)$ satisfies $V(I(\eta))\subseteq \lra{0}$. Let $v\in S\backslash\lra{0}$. 

If $v\notin \bigcup_{i=1}^s \ms{X}_i$, let $Q\in\mc{I}_{\ms{X}}$ be such that $Q(v)\neq 0$. Since $N_p\subseteq N_p^V$, if we write $\eta=\sum_{i=1}^\ell \alpha_i\dd x_i$, there exists $i\in\lra{1,\ldots,\ell}$ such that $\alpha_i(v)\neq 0$. Then $Q\dfrac{\dr }{\dr_{x_i}}\in\der(-\log \ms{X})$ and $\lrb{Q\frac{\dr }{\dr_{x_i}},\eta}=Q\alpha_i$ does not vanish on $v$, so that $v\notin V(I(\eta))$.

Let us assume that $v\in\bigcup_{i=1}^s\ms{X}_i$. Let $X=\displaystyle{\bigcap_{\ms{X}_i\ni v} \ms{X}_i}$.  We choose a basis $(y_1,\ldots,y_\ell)$ of $V^*$ such that $X$ is defined by $y_{m+1}=\dots=y_{\ell}=0$. Let $$J_1=\lra{i\in\lra{1,\ldots,s}\ ;\ X\not\subseteq \ms{X}_i} \text{ and } J_2=\lra{i\in\lra{1,\ldots,s}\ ;\ X\subseteq \ms{X}_i}=\lra{1,\ldots,s}\backslash J_1.$$

Let us denote $\ms{Z}_1=\bigcup_{i\in J_1} \ms{X}_i$ and $\ms{Z}_2=\bigcup_{i\in J_2} \ms{X}_i$. In particular, $v\notin\ms{Z}_1$.  Let $Q\in\mc{I}_{\ms{Z}_1}$ be such that $Q(v)\neq 0$. Let us prove that for all $i\in\lra{1,\ldots,m}$, we have $Q\frac{\dr }{\dr_{y_i}}\in\der(-\log \ms{X})$. 

Let $i\in\lra{1,\ldots,m}$. We have to prove that for all $h\in\mc{I}_{\ms{X}}$, $Q\frac{\dr}{\dr y_i}(h)\in\mc{I}_{\ms{X}}$. Since $\mc{I}_{\ms{X}}$ is radical, it suffices to prove that for all $j\in\lra{1,\ldots,s}$, the restriction of $Q\frac{\dr}{\dr y_i}(h)$ to $\ms{X}_j$ is zero. It is clear that for all $j\in J_1$, $Q\frac{\dr}{\dr_{y_i}}(h)\Big|_{\ms{X}_j}=0$ since $Q\in\mc{I}(\ms{Z}_1)$. Let $j\in J_2$. A generating family $(g_1,\ldots,g_k)$ of $\mc{I}_{\ms{X}_j}$ can be chosen so that for all $n\in\lra{1,\ldots,k}$, $g_n$ is homogeneous of degree $1$. Since $X\subseteq \ms{X}_j\subseteq \ms{X}$, we have $\mc{I}_{\ms{X}}\subseteq \mc{I}_{\ms{X}_j}\subseteq \mc{I}_{X}=\lrb{y_{m+1},\ldots,y_\ell}_S$. Therefore, for all $n\in\lra{1,\ldots,k}, g_n\in\lrb{y_{m+1},\ldots,y_\ell}_\C$ so that for all $n\in\lra{1,\ldots,k}$, $Q\frac{\dr }{\dr_{y_i}}(g_n)=0$. Thus, for all $h=\sum a_ng_n\in\mc{I}_{\ms{X}}$, $Q\frac{\dr }{\dr_{y_i}}(h)\Big|_{\ms{X}_j}=\lrp{\sum g_n \frac{\dr a_n}{\dr y_i}}\Big|_{\ms{X}_j}=0$. Hence the result: $Q\frac{\dr}{\dr_{y_i}}\in\der(-\log \ms{X})$. 

There exist $(\alpha_1,\ldots,\alpha_\ell)\in (S_p)^\ell$ such that $\eta=\sum_{n=1}^\ell \alpha_n\dd y_n$. Then for all $i\in\lra{1,\ldots,m}$, $\lrb{Q\frac{\dr}{\dr_{y_i}},\eta}=Q\alpha_i$. The restriction of $\eta$ to $X$ is:
$$r_{S,X}(\eta)=\overline{\alpha_1}\dd \overline{y_1}+\cdots+\overline{\alpha_m}\dd \overline{y_m}.$$ Since $\eta\in N_p^X$, it vanishes only at the origin. Therefore, there exists $i\in\lra{1,\ldots,m}$ such that $\alpha_i(v)=\overline{\alpha_i}(v)\neq 0$. Since $Q(v)\neq 0$, we have $v\notin V(I(\eta))$. 
\end{proof}

\begin{lem}
\label{lem:wtheta}
 Let $\w\in\Omega^q(\log \ms{X}/\ms{C})$ and $\theta\in\der(-\log \ms{X})$. Then $\lrb{\w,\theta}\in\Omega^{q-1}(\log \ms{X}/\ms{C})$. 
\end{lem}
\begin{proof}
By theorem~\ref{theo:carac:loga}, we can write:
$$g\w=\frac{\cx}{h}\wedge\xi+\lambda$$
with $\lambda\in\frac{1}{h}\mc{I}_\ms{C}\Omega^q$.
Then:
\begin{align*}
g\lrb{\w,\theta}&=\lrb{g\w,\theta}=\lrb{\frac{\cx}{h}\wedge\xi,\theta}+\lrb{\lambda,\theta}\\
                &=\lrb{\frac{\cx}{h},\theta}\wedge\xi+(-1)^k \frac{\cx}{h}\wedge\lrb{\xi,\theta}+\lrb{\lambda,\theta}
\end{align*}
Let us prove that $\lrb{\cx,\theta}\in\mc{I}_\ms{C}\Omega^{k-1}$. It suffices to prove that $\ov{\lrb{\cx,\theta}}\in\Omega^{k-1}\otimes\mathrm{Frac}(\co_\ms{C})$ is zero. We have:
\begin{align*}
\ov{\lrb{\cx,\theta}}&=\ov{\lrb{\beta_{\ms{X}/\ms{C}} \dd h_1\wedge\dots\wedge\dd h_k,\theta}}\\
                          &=\ov{\sum_{i=1}^k (-1)^{i-1} \beta_{\ms{X}/\ms{C}}\lrb{\dd h_i,\theta}\dd h_1\wedge\dots\wedge \widehat{\dd h_i}\wedge \dots\wedge\dd h_k }
\end{align*}
Since $\theta\in\der(-\log \ms{X})$, $\lrb{\dd h_i,\theta}\in\mc{I}_\ms{X}$. In addition, we deduce from the definition of $\beta_{\ms{X}/\ms{C}}$ that $\beta_{\ms{X}/\ms{C}}\lrb{\dd h_i,\theta}\Big|_{\ms{C}}=0$, since $\beta_{\ms{X}/\ms{C}}$ is zero on the components of $\ms{C}$ which are not in $\ms{X}$. 

Therefore, the form $\lrb{\w,\theta}$ satisfies theorem~\ref{theo:carac:loga} and $\lrb{\w,\theta}\in\Omega^{q-1}(\log \ms{X}/\ms{C})$. 
\end{proof}

\begin{remar}
The previous lemma is not specific to subspace arrangements. Let $X\subseteq (\C^\ell,0)$ be a reduced equidimensional subspace and $C$  a reduced complete intersection of the same dimension as $X$. Let $\der(-\log X)=\lra{\delta\in\Theta\ ;\ \delta(\mc{I}_X)\subseteq \mc{I}_X}$. With the same proof as for lemma~\ref{lem:wtheta}, we obtain that for all $\w\in\Omega^q(\log X/C)$ and $\theta\in\der(-\log X)$, $\lrb{\w,\theta}\in\Omega^{q-1}(\log X/C)$.  
\end{remar}

Thanks to lemmas~\ref{lem:eta:maximal:ideal} and~\ref{lem:wtheta}, \cite[Proposition 4.91]{orlik-terao-hyperplanes} can be generalized to subspace arrangements with a completely similar proof and gives:

\begin{prop}
\label{prop:cohom:finite}
If $\eta\in N_p$, then the cohomology groups of the $\eta$-complex are finite dimensional over $\C$. 
\end{prop}
%

By considering minimal integers $m,n\in\N$ such that $P(x,t)=x^n(1-x)^m\Psi(\Omega^\bullet(\log \ms{X}/\ms{C}),x,t)$ is a polynomial in $x$ and $t$ and using proposition~\ref{prop:cohom:finite}, one can prove in a similar way as for \cite[Proposition 4.133]{orlik-terao-hyperplanes}  that $m=0$, so that $\Psi(\Omega^\bullet(\log \ms{X}/\ms{C}),x,t)$ is a polynomial in $x$, $x^{-1}$ and $t$. 
\end{proof}

\subsection{Decomposition}

The purpose of this subsection is to distinguish in the $\Psi$-function of $\Omega^\bullet(\log \ms{X}/\ms{C})$ the contribution of the modules of multi-residues, which are intrinsic, from the contribution of the modules $\frac{1}{h}\mc{I}_\ms{C}\Omega^\bullet$ which depend on the choice of the complete intersection $\ms{C}$. 

\smallskip

Since $\mc{I}_\ms{C}$ is generated by a regular sequence, a free resolution of the modules $\frac{1}{h}\mc{I}_\ms{C}\Omega^q$ is deduced from the Koszul complex for all $q$, so that we can compute explicitly $\Psi\lrp{\frac{1}{h}\mc{I}_\ms{C}\Omega^\bullet,x,t}$. 
\begin{prop}
\label{prop:psi:wt:om}
We have:
$$\Psi\lrp{\frac{1}{h}\mc{I}_\ms{C}\Omega^\bullet,x,t}=t^\ell x^{-d} \lrp{1-\prod_{i=1}^k (1-x^{d_i})}.$$
\end{prop}
\begin{proof}
 Since for all $q\in\N$ the module $\Omega^q$ is a free $S$-module of rank $\binom{\ell}{q}$, we have: $$\Po\lrp{\frac{1}{h}\mc{I}_\ms{C}\Omega^q,x}=\binom{\ell}{q}\Po\lrp{\frac{1}{h}\mc{I}_\ms{C},x}.$$

A free resolution of $\frac{1}{h}\mc{I}_\ms{C}$ is deduced from the Koszul complex associated with the regular sequence $(h_1,\ldots,h_k)$. We recall that for all $i\in\lra{1,\ldots,k}$, $\deg(h_i)=d_i$ and $d=\deg(h)$. We have:

$$ 0\to S(0)\to \dots\to\bigoplus_{1\leqslant i_1\infe i_2\leqslant k} S(-d_{i_1}-d_{i_2}+d)  \to\bigoplus_{1\leqslant i_1\leqslant k} S(-d_{i_1}+d)\to\frac{1}{h}\mc{I}_\ms{C}\to 0.$$

Therefore:
\begin{align*}
\Po\lrp{\frac{1}{h}\mc{I}_\ms{C},x}&=\sum_{j=1}^k (-1)^{j-1} \sum_{1\leqslant i_1\infe\dots\infe i_j\leqslant k} x^{-d}\ \frac{x^{d_{i_1}+\dots+d_{i_j}}}{(1-x)^\ell}\\
 &= \frac{1}{x^d(1-x)^\ell} - \frac{1}{x^d(1-x)^\ell} \sum_{j=0}^k \sum_{1\leqslant i_1\infe \dots\infe i_j\leqslant k} (-1)^j x^{d_{i_1}+\dots+{d_{i_j}}}  \\
 &=\frac{1}{x^d(1-x)^\ell}-\frac{(1-x^{d_1})\dots (1-x^{d_k})}{x^d(1-x)^\ell}
\end{align*}
Thus:
\begin{align*}
\Psi\lrp{\frac{1}{h}\mc{I}_\ms{C}\Omega^\bullet,x,t}&=\sum_{q=0}^\ell \binom{\ell}{q}\Po\lrp{\frac{1}{h}\mc{I}_\ms{C},x}(t(1-x)-1)^q\\
 &=\Po\lrp{\frac{1}{h}\mc{I}_\ms{C},x} \sum_{q=0}^\ell \binom{\ell}{q}(t(1-x)-1)^q\\
 &=\lrp{\frac{1-(1-x^{d_1})\dots (1-x^{d_k})}{x^d(1-x)^\ell}}(t(1-x))^\ell\\
 &=t^\ell\ \lrp{\frac{1-(1-x^{d_1})\dots (1-x^{d_k})}{x^d}}
\end{align*}
Hence the result.
\end{proof}

\begin{prop}
\label{prop:psi:wt:rc}
The $\Psi$-functions of $\Omega^\bullet(\log \ms{X}/\ms{C})$, $\mc{R}_\ms{X}^\bullet$ and $\frac{1}{h}\mc{I}_\ms{C}\Omega^\bullet$ are related as follows:
\begin{align*}
\Psi(\Omega^\bullet(\log \ms{X}/\ms{C}),x,t)&=\Psi\lrp{\frac{1}{h}\mc{I}_\ms{C}\Omega^\bullet,x,t}+(t(1-x)-1)^k x^{-k}\Psi(\mc{R}_\ms{X}^\bullet,x,t)
\end{align*}
\end{prop}

\begin{proof}
We recall the exact sequence~\eqref{eq:wtilde:log:rx} which holds for all $q\in\N$:
$$0\to \frac{1}{h}\mc{I}_\ms{C}\Omega^q\to \Omega^q(\log \ms{X}/\ms{C})\xrightarrow{\resxc}\mc{R}_\ms{X}^{q-k}\to 0.$$

By proposition~\ref{prop:resxc:deg}, the map $\resxc$ is homogeneous of degree $k$. Since Hilbert-Poincar\'e series are additive, we have for all $q\in\N$:
$$\Po(\Omega^q(\log \ms{X}/\ms{C}),x)=\Po\lrp{\frac{1}{h}\mc{I}_\ms{C}\Omega^q, x}+x^{-k} \Po(\mc{R}_\ms{X}^{q-k},x).$$

Therefore:
\begin{align*}
\Psi(\Omega^\bullet(\log \ms{X}/\ms{C}),x,t)&=\sum_{q=0}^\ell \Po(\Omega^q(\log \ms{X}/\ms{C}),x)(t(1-x)-1)^q\\
       &=\sum_{q=0}^\ell \Po\lrp{\frac{1}{h}\mc{I}_\ms{C}\Omega^q,x}(t(1-x)-1)^q+\sum_{q=k}^\ell x^{-k}\Po(\mc{R}_\ms{X}^{q-k},x)(t(1-x)-1)^q\\
       &=\Psi\lrp{\frac{1}{h}\mc{I}_\ms{C}\Omega^\bullet,x,t}+x^{-k}(t(1-x)-1)^k\Psi(\mc{R}_\ms{X}^\bullet,x,t)
\end{align*} 
Hence the result.
\end{proof}

\begin{remar}

If $(h_1',\ldots,h_k')\subseteq \ix$ is another homogeneous regular sequence defining a reduced complete intersection $\ms{C}'$, then the two functions $\Psi(\Omega^\bullet(\log \ms{X}/\ms{C}),x,t)$ and $\Psi(\Omega^\bullet(\log \ms{X}/\ms{C}'),x,t)$ may be different. 

Indeed, if the degrees of the equations of $\ms{C}'$ are different from the degrees of the equations of $\ms{C}$, then by proposition~\ref{prop:psi:poly}, $\Psi(\frac{1}{h}\mc{I}_\ms{C}\Omega^\bullet,x,t)$ may be different from $\Psi(\frac{1}{h'}\mc{I}_{\ms{C}'}\Omega^\bullet,x,t)$. On the contrary, the function $\Psi(\mc{R}_\ms{X}^\bullet,x,t)$ does not depend on the choice of the complete intersection $\ms{C}$. 
\end{remar}

\subsection{An explicit computation: the case of complete intersection line arrangements}

Let us give an example which is the motivation of the next part of this paper. The computation of the modules of multi-logarithmic forms along a reduced equidimensional subspace is difficult in general. The case of quasi-homogeneous complete intersection curves is explicitly computed in \cite[Th\'eor\`eme 6.1.29 and Th\'eor\`eme 6.1.33]{polthese}.

\begin{prop}
\label{prop:hom:ic:curves:psi}
Let $\ms{C}\subseteq \C^\ell$ be a reduced homogeneous\footnote{It is not necessary to assume that  $\ms{C}$ is a line arrangement.} complete intersection curve with embedding dimension $\ell$. We suppose that $\mc{I}_\ms{C}$ is generated by a homogeneous regular sequence $(h_1,\ldots,h_{\ell-1})$. We set $h=h_1\cdots h_{\ell-1}$.  For all $i\in\lra{1,\ldots,\ell-1}$, we denote by $d_i$ the degree of $h_i$, and $d=\deg(h)$. Then:

\begin{multline}
\label{eq:psi:ci:line}
\Psi(\Omega^\bullet(\log \ms{C}),x,t)=t^\ell x^{-d}\lrp{1-\prod_{i=1}^{\ell-1} (1-x^{d_i})}\\+(t(1-x)-1)^{\ell-1}x^{-\ell+1}\lrp{1+(t-1)x^{\ell-d-1}\prod_{i=1}^{\ell-1} (1+x+\cdots+x^{d_i-1})}.
\end{multline}

In particular, $$\Psi(\Omega^\bullet(\log \mc{C}),1,t)=t^\ell+(-1)^{\ell-1} (d_1\cdots d_{\ell-1}) t+ (-1)^{\ell}(d_1\cdots d_{\ell-1}-1).$$
\end{prop}

In particular, if $\ms{C}$ is a line arrangement, then $\ms{C}$ has $d_1\cdots d_{\ell-1}$ components and we obtain the following property which generalizes Solomon-Terao formula~\ref{theo:solomon:terao}:

\begin{cor}
\label{cor:line:ic:psi}
If in addition the reduced complete intersection $\ms{C}$ is a line arrangement, the characteristic polynomial of $\ms{C}$ is:
$$\chi(\ms{C},t)=t^\ell-d_1\cdots d_{\ell-1} t+(d_1\cdots d_{\ell-1}-1).$$

We thus have:
$$\chi(\ms{C},t)=t^\ell-\Psi(\mc{R}_\ms{C}^\bullet,1,t).$$

In particular, if $\ell-1$ is odd, $$\Psi(\Omega^\bullet(\log \ms{C}),1,t)=\chi(\ms{C},t).$$
\end{cor}
\begin{proof}[Proof of proposition~\ref{prop:hom:ic:curves:psi}]
Thanks to proposition~\ref{prop:psi:wt:rc} and \ref{prop:psi:wt:om}, it is sufficient to compute the $\Psi$-function relative to the module of logarithmic multi-residues. Let us compute $\Po(\mc{R}_\ms{C},x)$ and $\Po(\mc{R}_\ms{C}^1,x)$. 

By \cite[Th\'eor\`eme 6.1.29]{polthese}, a free resolution of $\mc{R}_\ms{C}$ is given by:
\begin{multline}      
  0\to S(-(\ell-1))^{\binom{\ell}{\ell-1}}      \to \bigoplus_{1\leqslant i_1\infe\dots\infe i_{\ell-2}\leqslant \ell-1} S(d-\ell-d_{i_1}-\cdots-d_{i_{\ell-2}})\oplus S(-(\ell-2))^{\binom{\ell}{\ell-2}}  \to  \dots  \\ 
  \dots \to \bigoplus_{1\leqslant i_1\leqslant \ell-1} S(d-\ell-d_{i_1})\oplus S(-1)^{\binom{\ell}{1}} \to S(d-\ell) \oplus S(0)\to \mc{R}_\ms{C}\to 0.
\end{multline}

We then have:
\begin{align*}
\Po(\mc{R}_\ms{C},x)&=\sum_{j=0}^{\ell-1} (-1)^j \binom{\ell}{j}\frac{x^j}{(1-x)^\ell}+\sum_{j=0}^{\ell-2} (-1)^j x^{\ell-d} \sum_{1\leqslant i_1\infe \dots\infe i_j\leqslant \ell-1} \frac{x^{d_{i_1}+\cdots+d_{i_j}}}{(1-x)^\ell}\\
  &= (-1)^{\ell-1} \frac{x^\ell}{(1-x)^\ell} +\frac{1}{(1-x)^\ell}\sum_{j=0}^{\ell}  \binom{\ell}{j}(-x)^j\\ 
  &\ \ + (-1)^\ell x^{\ell-d}\ \frac{x^d}{(1-x)^\ell}+\frac{x^{\ell-d}}{(1-x)^\ell}\sum_{j=0}^{\ell-1} \sum_{1\leqslant i_1\infe\dots\infe i_j\leqslant \ell-1} (-1)^j (x^{d_{i_1}}+\cdots+x^{d_{i_j}})\\
  &=1 + \frac{x^{\ell-d}}{(1-x)^\ell} (1-x^{d_1})\cdots (1-x^{d_{\ell-1}}).
\end{align*}

Let us compute $\Po(\mc{R}_\ms{C}^1,x)$. We recall that $\mc{R}_\ms{C}^1=\mathrm{res}_{\ms{C}}\lrp{\Omega^{\ell}(\log \ms{C})}$. In addition, it is easy to see that $\Omega^\ell(\log \ms{C})=\frac{1}{h}\Omega^m$ is a free $S$-module of rank $1$. Therefore, $\mc{R}_\ms{C}^1$ is the free $\co_\ms{C}$-module generated by $\mathrm{res}_{\ms{C}}\lrp{\frac{\dd x_1\wedge \dots\wedge \dd x_\ell}{h}}$. A free resolution of $\mc{R}_\ms{C}^1$ can therefore be deduced from the Koszul complex associated with $(h_1,\ldots,h_{\ell-1})$:

\begin{multline}       0\to S(-\ell+1) \to \bigoplus_{1\leqslant i_1\infe\dots\infe i_{\ell-2}\leqslant \ell-1} S(d-\ell+1-d_{i_1}-\cdots -d_{i_{\ell-2}})\to \dots \\
\dots \to \bigoplus_{1\leqslant i_1\leqslant \ell-1} S(d-\ell+1-d_{i_1})\to S(d-\ell+1) \to \mc{R}_\ms{C}^1 \to 0.
\end{multline}

Therefore, we have:
\begin{align*}
\Po(\mc{R}_\ms{C}^1,x)&=\sum_{j=0}^{\ell-1} (-1)^j \sum_{1\leqslant i_1\infe \dots \infe i_j\leqslant \ell-1} x^{\ell-d-1}\ \frac{x^{d_{i_1}}+\cdots+x^{d_{i_j}}}{(1-x)^\ell}\\
 &=\frac{x^{\ell-d-1}}{(1-x)^\ell}\ {(1-x^{d_1})\cdots (1-x^{d_{\ell-1}})}
\end{align*}

We then  compute the $\Psi$-function associated with the modules of logarithmic multi-residues:

\begin{align*}
\Psi(\mc{R}_\ms{C}^\bullet,x,t)&=\Po(\mc{R}_\ms{C},x)+(t(1-x)-1)\Po(\mc{R}_\ms{C}^1,x)\\
  &= 1 + \frac{x^{\ell-d}}{(1-x)^\ell} (1-x^{d_1})\cdots (1-x^{d_{\ell-1}})+ (t(1-x)-1)\frac{x^{\ell-d-1}}{(1-x)^\ell}\ {(1-x^{d_1})\cdots (1-x^{d_{\ell-1}})}\\
  &=1+\frac{x^{\ell-d-1}}{(1-x)^\ell}(1-x^{d_1})\cdots(1-x^{d_{\ell-1}})(x+t(1-x)-1)\\
  &=1+\frac{x^{\ell-d-1}}{(1-x)^\ell}(1-x^{d_1})\cdots (1-x^{d_{\ell-1}})(1-x)(t-1)\\
 &= 1+ x^{\ell-d-1}(t-1)\prod_{i=1}^{\ell-1} (1+x+\cdots+x^{d_i-1})
\end{align*}

This computation and propositions~\ref{prop:psi:wt:om} and \ref{prop:psi:wt:rc} show that $\Psi(\Omega^\bullet(\log \ms{C}),x,t)$ is given by~\eqref{eq:psi:ci:line}.
\end{proof}

\section{Solomon-Terao formula for subspace arrangements}
\label{sec:solomon:terao}

We investigate in this section a generalization of Solomon-Terao formula~\ref{theo:solomon:terao} for subspace arrangements, which leads to our main theorem~\ref{theo:subspace:solomon:terao}.

\subsection{Statement of the main theorem}

 We assume that $\ms{X}=\lra{\ms{X}_1,\ldots,\ms{X}_s}$ is a reduced  equidimensional subspace arrangement of codimension $k$ in $\C^\ell$, and that $\ms{C}$ is a reduced complete intersection subspace arrangement of codimension $k$ which contains $\ms{X}$. We denote by $(h_1,\ldots,h_k)$ an homogeneous regular sequence defining $\ms{C}$, and $h=h_1\cdots h_k$. We recall that $d_i$ denotes the degree of $h_i$ for all $i\in\lra{1,\ldots,k}$. 

As suggested by the case of complete intersection line arrangements, we will distinguish the part of the $\Psi$-function of $\Omega^\bullet(\log \ms{X}/\ms{C})$ which comes from $\mc{R}_\ms{X}^\bullet$ from the part which comes from $\frac{1}{h}\mc{I}_\ms{C}\Omega^\bullet$. 
We will use the following notations:
\begin{nota}
Let $Y\in L(\ms{X})$. We denote $L_Y=\lra{Z\in L(\ms{X})\ ;\ Y\subseteq Z}$. If $Y\in L(\ms{X})$, we set $\ms{X}_Y=\lra{\ms{X}_i\ ;\ Y\subseteq \ms{X}_i}\subseteq \ms{X}$.
\end{nota}

Our main theorem is:

\begin{theo}
\label{theo:subspace:solomon:terao}
Let $\ms{X}$ be a reduced equidimensional subspace arrangement in $\C^\ell$ with codimension~$k$. Let $\ms{C}$ be a reduced complete intersection subspace arrangement of codimension $k$ containing $\ms{X}$. If for all $Y\in L(\ms{X})\backslash\lra{V}$ we have $\Psi(\mc{R}_{\ms{X}_Y}^\bullet,1,1)=1$ then for all $Y\in L(\ms{X})$, we have:
\begin{equation}
\label{eq:theo}
\begin{aligned}
\chi(\ms{X}_Y,t)&=\Psi\lrp{\frac{1}{h}\mc{I}_\ms{C}\Omega^\bullet,1,t}-\Psi(\mc{R}_{\ms{X}_Y}^\bullet,1,t)\\
  &=t^\ell-\Psi(\mc{R}_{\ms{X}_Y}^\bullet,1,t).
  \end{aligned}
\end{equation}

In particular, if $k$ is odd, $$\chi(\ms{X}_Y,t)=\Psi(\Omega^\bullet(\log \ms{X}_Y/\ms{C}),1,t).$$
\end{theo}

\begin{remar}We will prove in section~\ref{sec:ex} that the condition $\Psi(\mc{R}_{\ms{X}_Y},1,1)=1$ is always satisfied for any line arrangement of any codimension, but is not necessarily satisfied for subspace arrangements of higher dimension.
\end{remar}

\begin{remar}
From remark~\ref{remar:wq:degree}, one can notice that equation~\eqref{eq:theo} can be rephrased as $$\chi(\ms{X}_Y,t)=t^\ell-\Psi(\w_{\ms{X}_Y}^\bullet,1,t)$$ since $\Psi(\w_{\ms{X}_Y}^\bullet,x,t)=x^k\Psi(\mc{R}_{\ms{X}_Y}^\bullet,x,t)$.
\end{remar}

\subsection{Proof of theorem~\ref{theo:subspace:solomon:terao}}

Let us introduce the following function.
\begin{nota}
\label{nota:wt:psi}
For $Y\in L(\ms{X})$, we define:

$$\wt{\Psi}(\ms{X}_Y,x,t)=\Psi\lrp{\frac{1}{h}\mc{I}_\ms{C}\Omega^\bullet,x,t}+(-1)^{k+1}(t(1-x)-1)^kx^{-k}\Psi(\mc{R}_{\ms{X}_Y}^\bullet,x,t)\in\Z[x,x^{-1},t].$$

In particular, $\wt{\Psi}(\ms{X}_Y,1,t)=\Psi\lrp{\frac{1}{h}\mc{I}_\ms{C}\Omega^\bullet,1,t}-\Psi(\mc{R}_{\ms{X}_Y}^\bullet,1,t)=t^\ell-\Psi(\mc{R}_{\ms{X}_Y}^\bullet,1,t)$, and if $k$ is odd, we have $\wt{\Psi}(\ms{X}_Y,x,t)=\Psi(\Omega^\bullet(\log \ms{X}_Y/\ms{C}),x,t)$. 
\end{nota}

The proof of \cite[Proposition 4.135]{orlik-terao-hyperplanes} can be extended to subspace arrangements, so that we have:
\begin{prop}[\protect{\cite[Proposition 4.135]{orlik-terao-hyperplanes}}]
\label{prop:char:chi}
If $G : L(\ms{X})\to \Z[t]$ is a map satisfying the following four conditions:
\begin{enumerate}[a)]
\item \label{char:chi:1} $G(V)=t^\ell$,
\item \label{char:chi:2} for $Y\neq V$, $G(Y)|_{t=1}=0$,
\item \label{char:chi:3} for all $Y\in L(\ms{X})$, $t^{\dim Y}$ divides $G(Y)$,
\item \label{char:chi:4} for all $Y\in L(\ms{X})$, $\displaystyle{\deg_t\lrp{\sum_{Z\in L_Y}\mu(Z,Y)G(Z)}\leqslant \dim Y}$,
\end{enumerate}
then for all $Y\in L(\ms{X})$, $G(Y)=\chi(\ms{X}_Y,t)$. 
\end{prop}

\begin{nota}
For $Y\in L(\ms{X})$, we set $G(Y)=\wt{\Psi}(\ms{X}_Y,1,t)\in\Z[t]$. 
\end{nota}

To prove theorem~\ref{theo:subspace:solomon:terao}, it suffices to prove that $G$ satisfies the properties of proposition~\ref{prop:char:chi}. Let us study each property.

\subsubsection{Property~\eqref{char:chi:1}}

For the empty arrangement, $\mc{R}_{\emptyset}^q=0$ for all $q\in\N$ so that, using proposition~\ref{prop:psi:wt:om} we obtain the first property~\eqref{char:chi:1}: $$G(V)=\Psi\lrp{\frac{1}{h}\mc{I}_\ms{C}{\Omega^\bullet},1,t}=t^\ell.$$

\begin{remar}
One can notice that for $q\in\N$, $\Omega^q(\log \emptyset/\ms{C})=\frac{1}{h}\mc{I}_\ms{C}\Omega^q$  whereas  $\Omega^q(\log \emptyset)=\Omega^q$. 
\end{remar}
\subsubsection{Property \eqref{char:chi:2}}

Since $\Psi\lrp{\frac{1}{h}\mc{I}_{\ms{C}}\Omega^\bullet,1,1}=1$, the assumption $\Psi(\mc{R}_{\ms{X}_Y}^\bullet,1,1)=1$ for all $Y\in L(\ms{X})\backslash\lra{V}$ ensures that $G(Y)|_{t=1}=0$ for all $Y\in L(\ms{X})$, which gives us property~\eqref{char:chi:2}. 

\begin{remar}
For an hyperplane arrangement $\ms{A}$, property~\eqref{char:chi:2} is always satisfied as it is proved in \cite[Proposition 4.132]{orlik-terao-hyperplanes}. The proof of this result relies on the fact that if $\alpha$ is a reduced equation of one of the hyperplanes of $\ms{A}$, the complex $(\Omega^\bullet(\log\ms{A}), \frac{\dd \alpha}{\alpha}\wedge)$ is acyclic. In codimension at least $2$, the modules of multi-logarithmic differential forms are no more stable under the exterior product. In addition, example~\ref{ex:xy:zt} shows that there exist equidimensional reduced subspace arrangements such that this condition is not satisfied.
\end{remar}

\subsubsection{Property \eqref{char:chi:3}}
Let us generalize \cite[Proposition 4.134]{orlik-terao-hyperplanes}. 

 Since $t^{\dim(Y)}$ divides $t^\ell$, it is sufficient to prove that $t^{\dim(Y)}$ divides $\Psi(\mc{R}_{\ms{X}_Y}^\bullet,1,t)$.

We will need the following property, which can be compared with \cite[Proposition 4.84]{orlik-terao-hyperplanes}. 

\begin{prop}
\label{prop:prod:formes}
Let $n\in\N$, $n\neq 0$.
Let $\ms{Z}$ be an equidimensional subspace arrangement of codimension $k$ in $V_1=\C^n$ and let $\ms{C}_0$ be a reduced complete intersection subspace arrangement of codimension $k$ in $\C^n$ containing $\ms{Z}$. Let $m\in\N$, $m\geqslant 1$ and $V_2=\C^m$. Then $\ms{C}_0\times V_2$ is a reduced complete intersection subspace arrangement which contains $\ms{Z}\times V_2$ and for all $q\in\N$, with $\Omega^q[V_2]$ the module of differential forms on $V_2$, we have:
$$\Omega^q(\log (\ms{Z}\times V_2/\ms{C}_0\times V_2))=\sum_{j=0}^q \Omega^j(\log \ms{Z}/\ms{C}_0)\otimes \Omega^{q-j}[V_2].$$

\end{prop}
\begin{proof}
We set $S^1=\C[x_1,\ldots,x_n]$ and $S^2=\C[y_1,\ldots,y_m]$. We identify any $f\in S^1$ with $f\otimes 1\in S^1\otimes_\C S^2$. Let $(h_1,\ldots,h_k)\subseteq S^1$ be an homogeneous regular sequence defining $\ms{C}_0$. Then $(h_1\otimes 1,\ldots, h_k\otimes 1)\in S^1\otimes S^2$ is an homogeneous regular sequence which defines the reduced complete intersection $\ms{C}'=\ms{C}_0\times V_2$.  

\smallskip

Let $\w\in\frac{1}{h}\Omega^q[V_1\times V_2]$. Then $\w\in\Omega^q\lrp{\log \lrp{\ms{Z}\times V_2/\ms{C}'}}$ if and only if for all $f\in \mc{I}_\ms{Z}\otimes S^2$,  $f\w\in\frac{1}{h}\lrp{\mc{I}_{\ms{C}_0}\otimes S^2}\Omega^q[V_1\times V_2]$ and $\dd f\wedge\w\in\frac{1}{h}\lrp{\mc{I}_{\ms{C}_0}\otimes S^2}\Omega^{q+1}[V_1\times V_2]$. Since $\mc{I}_\ms{Z}\otimes S^2$ is generated as an ideal in $S^1\otimes S^2$ by $\mc{I}_\ms{Z}\otimes 1$, it is sufficient to consider only elements in $\mc{I}_\ms{Z}\otimes 1$, which will be denoted by $\mc{I}_\ms{Z}$. 

\medskip

We can write $\w=\sum_{j=0}^q \w_j$ where $\w_j\in\frac{1}{h}\Omega^j[V_1]\otimes \Omega^{q-j}[V_2]$. Then $\w\in\Omega^q(\log\lrp{\ms{Z}\times V_2/\ms{C}'})$ if and only if for all $j\in\lra{0,\ldots,q}$, $\w_j\in\Omega^q(\log \lrp{\ms{Z}\times V_2/\ms{C}')}$.  

\medskip

 Let us write $\w_j=\sum_{p=1}^r \eta_p\wedge \xi_p$ where for $1\leqslant p\leqslant r$, $\eta_p\in\frac{1}{h}\Omega^j[V_1]$ and $\xi_p\in\Omega^{q-j}[V_2]$ and the $\xi_p$ are linearly independant over $\C$. Let us prove as in \cite[Lemma 4.83]{orlik-terao-hyperplanes} that  $\w_j\in\Omega^q(\log \lrp{\ms{Z}\times V_2/\ms{C}'})$ if and only if for all $p$, $\eta_p\in\Omega^j(\log\ms{Z}/\ms{C}_0)$.

\medskip 
 
  Let us assume that for all $p\in\lra{1,\ldots,q}$, $\eta_p\in\Omega^j(\log\ms{Z}/\ms{C}_0)$. Then for all $f\in\mc{I}_\ms{Z}$, $f\eta_p\in\frac{1}{h}\mc{I}_{\ms{C}_0}\Omega^{j}[V_1]$ and $\dd f\wedge\eta_p\in\frac{1}{h}\mc{I}_{\ms{C}_0}\Omega^{j+1}[V_1]$. Therefore, $f\lrp{\sum \eta_p\wedge \xi_p}\in\frac{1}{h}\mc{I}_{\ms{C}_0}\Omega^j[V_1]\otimes\Omega^{q-j}[V_2]$ and $\dd f\wedge\lrp{ \sum \eta_p\wedge \xi_p}\in\frac{1}{h}\mc{I}_{\ms{C}_0}\Omega^{j+1}[V_1]\otimes \Omega^{q-j}[V_2]$ and $\w_j\in\Omega^q(\log \lrp{\ms{Z}\times V_2/\ms{C}'})$. Thus $\w\in\Omega^q(\log (\ms{Z}\times V_2/\ms{C}'))$. 
 
\medskip 
 
 Let us assume now that $\w\in\Omega^q(\log (\ms{Z}\times V_2/\ms{C}'))$, or equivalently that for all $j\in\lra{0,\ldots,q}$, $\w_j\in\Omega^q(\log (\ms{Z}\times V_2/\ms{C}'))$. Let $f\in\mc{I}_\ms{Z}$. We have 
$$f \lrp{\sum \eta_p\wedge\xi_p}=\sum (f\eta_p)\wedge\xi_p\in \frac{1}{h}\mc{I}_{\ms{C}_0}\Omega^j[V_1]\otimes \Omega^{q-j}[V_2]$$ 
 $$\dd f\wedge\lrp{\sum \eta_p\wedge \xi_p}=\sum_p (\dd f\wedge \eta_p)\wedge \xi_p\in\frac{1}{h}\mc{I}_{\ms{C}_0}\Omega^{j+1}[V_1]\otimes \Omega^{q-j}[V_2].$$
 Therefore, since the $\xi_p $ are linearly independant, we have for all $p$, $f\eta_p\in\frac{1}{h}\mc{I}_{\ms{C}_0}\Omega^j[V_1]$ and $\dd f\wedge\eta_p\in\frac{1}{h}\mc{I}_{\ms{C}_0}\Omega^{j+1}[V_1]$ so that $\eta_p\in\Omega^j(\log \ms{Z}/\ms{C}_0)$. Hence the result.
\end{proof}

Let $Y\in L(\ms{X})$, $Y\neq \C^\ell$. Let $m$ be the dimension of $Y$. Then there exists $u\in\lra{1,\ldots,s}$ and $1\leqslant i_1\infe\dots\infe i_u\leqslant s$ such that $\ms{X}_Y=\ms{X}_{i_1}\cup\dots\cup \ms{X}_{i_u}$. Since all the irreducible components of $\ms{X}_Y$ contains $Y$, there exists an equidimensional  subspace arrangement $\ms{Z}=\ms{Z}_1\cup\dots\cup\ms{Z}_u$ of codimension $k$ in $\C^{\ell-m}$ such that $\ms{X}_Y=\ms{Z}\times Y=(\ms{Z}_1\times Y)\cup\dots\cup (\ms{Z}_u\times Y)$. Let $\ms{C}_0$ be a reduced complete intersection subspace arrangement of codimension $k$ in $\C^{\ell-m}$ containing $\ms{Z}$.

Let us consider $\Psi(\Omega^\bullet(\log (\ms{Z}\times Y/\ms{C}_0\times Y)),x,t)$. One can notice that this function may be different from $\Psi(\Omega^\bullet(\log( \ms{Z}\times Y/\ms{C})),x,t)$ since $\ms{C}$ and $\ms{C}':=\ms{C}_0\times Y$ may be different. However, since the modules of multi-residues are intrinsic, the $\Psi$-function relative to $\mc{R}_{\ms{Z}\times Y}^\bullet$ does not depend on the choice of the complete intersection.

Let $q\in\N$. Since $\Omega^q(\log \ms{X}_Y/\ms{C}')=\sum_{j=0}^q \Omega^j(\log \ms{Z}/\ms{C}_0)\otimes \Omega^{q-j}[Y]$, we have:
$$\Po(\Omega^q(\log \ms{X}_Y/\ms{C}'),x)=\sum_{j=0}^q \Po(\Omega^j(\log \ms{Z}/\ms{C}_0),x)\Po(\Omega^{q-j}[Y],x)$$

Therefore, we have:
{\small
\begin{align*}
\Psi(\Omega^\bullet(\log \ms{X}_Y/\ms{C}'),x,t)&=\sum_{q=0}^\ell\sum_{j=0}^q \Po(\Omega^j(\log \ms{Z}/\ms{C}_0),x)(t(1-x)-1)^j\Po(\Omega^{q-j}[Y],x)(t(1-x)-1)^{q-j}\\
 &= \lrp{\sum_{p=0}^{\ell-m} \Po(\Omega^p(\log \ms{Z}/\ms{C}_0),x)(t(1-x)-1)^p}\cdot\lrp{\sum_{q=0}^{m} \Po(\Omega^{q}[Y],x)(t(1-x)-1)^q}\\
 &=\Psi(\Omega^\bullet(\log \ms{Z}/\ms{C}_0),x,t)\cdot\Psi(\Omega^\bullet[Y],x,t)
\end{align*}}

By \cite[Proposition 4.131]{orlik-terao-hyperplanes}, we have $\Psi(\Omega^\bullet[Y],x,t)=t^m$ since for all $q\in\N$, the module of logarithmic forms\footnote{in the sense of hypersurfaces.} of the empty arrangement in $Y$ is $\Omega^q[Y]$. Therefore, $t^m$ divides $\Psi(\Omega^\bullet(\log\ms{X}_Y/\ms{C}'),x,t)$. We deduce from propositions~\ref{prop:psi:wt:rc} and~\ref{prop:psi:wt:om} that $t^m$ also divides $\Psi(\mc{R}_{\ms{X}_Y}^\bullet,x,t)$, so that $t^m$ divides $\wt{\Psi}(\ms{X}_Y,x,t)$. Therefore, $t^m$ divides $G(Y)$ which gives us property~\eqref{char:chi:3}.

\subsubsection{Property \eqref{char:chi:4}}

Our purpose is to prove that for all $Y\in L(\ms{X})$, $$\displaystyle{\deg_t\lrp{\sum_{{Z\in L_Y}}\mu(Z,Y)G(Z)}\leqslant \dim Y}$$
where $L_Y=\lra{Z\in L(\ms{X})\ ;\ Z\leqslant Y}$.

Since for all $Z\in L(\ms{X})$, $G(Z)=\Psi\lrp{\frac{1}{h}\mc{I}_\ms{C}\Omega^\bullet,1,t}-\Psi(\mc{R}_{\ms{X}_Z}^\bullet,1,t)$ where $\Psi\lrp{\frac{1}{h}\mc{I}_\ms{C}\Omega^\bullet,1,t}$ does not depend on $Z$, and since by \cite[Lemma 2.38]{orlik-terao-hyperplanes}, $\sum_{{Z\in L_Y}}\mu(Z,Y)=0$, we have $$\sum_{{Z\in L_Y}}\mu(Z,Y)G(Z)=\sum_{{Z\in L_Y}}\mu(Z,Y)\Psi(\mc{R}_{\ms{X}_Z}^\bullet,1,t)=\sum_{{Z\in L_Y}}\mu(Z,Y)\Psi(\Omega^\bullet{\log (\ms{X}_Z/\ms{C})},1,t).$$

Therefore, it is equivalent to prove that property~\eqref{char:chi:4} is satisfied for the modules of multi-logarithmic forms. 
\begin{nota}
Let $Y\in L(\ms{X})$ and $\p\in\mathrm{Spec}(S)$. We recall that $\ms{X}_Y$ denotes the subarrangement composed of the components of $\ms{X}$ which contain $Y$. We set $$Y(\p)=\bigcap_{\substack{\ms{X}_i\in \ms{X}_Y\\ \mc{I}_{\ms{X}_i}\subseteq \p}} \ms{X}_i$$

We then have $Y\subseteq Y(\p)$, so that $\ms{X}_{Y(\p)}\subseteq \ms{X}_Y$. 
\end{nota}

\begin{de}[\protect{\cite[Definition 4.121]{orlik-terao-hyperplanes}}]
A covariant functor $F : L(\ms{X}) \to (S\text{-mod})$ is called local if for all $\p\in\mathrm{Spec}(S)$ and for all $Y\in L(\ms{X})$, the localization at $\p$ of the map $\nu_{Y(\p),Y} : F(Y(\p)) \to F(Y)$ is an isomorphism.
\end{de}
The following proposition is given for hyperplane arrangements, but it can be proved for subspace arrangements with exactly the same proof. 
\begin{prop}[\protect{\cite[Theorem 4.128]{orlik-terao-hyperplanes}}]\label{prop:local-functor-mu}
Let $Y\in L(\ms{X})$. If $F : L(\ms{X})\to (S\text{-mod})$ is a local covariant functor, then $$\sum_{Z\in L_Y} \mu(Z,Y)\Po(F(Z),x)$$
has a pole at $x=1$ of order at most $\dim Y$. 
\end{prop}

Let us prove the following proposition:

\begin{prop}
\label{prop:loc:func:res}
Let $q\in\N$. The functor $F_q : L(\ms{X})\to (S\text{-mod})$ defined for $Y\in L(\ms{X})$ by $F_q(Y)=\Omega^q(\log \ms{X}_Y/\ms{C})$ is a local covariant functor. 
\end{prop}
\begin{proof}
Let us prove first that $F_q$ is a covariant functor. Let $Y_1\leqslant Y_2$. Then $\ms{X}_{Y_1}\subseteq \ms{X}_{Y_2}$, so that $\mc{I}_{\ms{X}_{Y_2}}\subseteq \mc{I}_{\ms{X}_{Y_1}}$. From the definition of the modules of multi-logarithmic forms, it is easy to see that $\Omega^q(\log \ms{X}_{Y_1}/\ms{C})\subseteq \Omega^q(\log \ms{X}_{Y_2}/\ms{C})$, which shows that the functor $F_q$ is covariant. 

Let us prove that the functor $F_q$ is local. Let $\p\in\mathrm{Spec}(S)$. 

Let $Y\in L(\ms{X})$. Since $Y(\p)\leqslant Y$, we have $\Omega^q(\log \ms{X}_{Y(\p)}/\ms{C})\subseteq \Omega^q(\log \ms{X}_{Y}/\ms{C})$. Since localization is an exact functor, we have:
\begin{equation}
\label{eq:local:func}\lrp{\Omega^q(\log \ms{X}_{Y(\p)}/\ms{C})}_\p\subseteq \lrp{\Omega^q(\log \ms{X}_{Y}/\ms{C})}_\p.
\end{equation}

Let us prove that we have an equality. If $\ms{X}_{Y(\p)}=\ms{X}_Y$, the equality is clear. Let us assume that $\ms{X}_{Y(\p)}\neq \ms{X}_Y$. Let us denote by $\ms{X}'$ the union of the components of $\ms{X}_Y$ which are not in $\ms{X}_{Y(\p)}$. In particular, $\mc{I}_{\ms{X}'}\not\subseteq \p$. Let $Q\in \mc{I}_{\ms{X}'}$ be such that $Q\notin\p$. Then for all component $\ms{X}_i$ of $\ms{X}_{Y(\p)}$, $Q\notin\mc{I}_{\ms{X}_i}$. 

Let $\w\in\Omega^q(\log \ms{X}_{Y}/\ms{C})$. Let us prove that $Q\w\in\Omega^q(\log \ms{X}_{Y(\p)}/\ms{C})$, which will give us the equality in~\eqref{eq:local:func}.

Let $c_{\ms{X}_Y/\ms{C}}$ be the fundamental form of $\ms{X}_Y$ (see definition~\ref{de:fund:class}). Let us denote by $\ms{Z}$ the union of the components of $\ms{C}$ which are not in $\ms{X}_Y$.  By theorem~\ref{theo:carac:loga}, there exists $g\in S$ inducing a non zero divisor in $\co_{\ms{C}}$, $\xi\in\Omega^{q-k}$ and $\eta\in\frac{1}{h}\mc{I}_{\ms{C}}\Omega^q$ such that:
$$g\w=\frac{c_{\ms{X}_Y/\ms{C}}}{h}\wedge \xi+\eta.$$

In particular, from the definition of $c_{\ms{X}_Y/\ms{C}}$ and since $g$ induces a non zero divisor in $\co_{\ms{C}}$, one can see that $\w\in \frac{1}{h}\mc{I}_{\ms{Z}}\Omega^q$. Let $\ms{Z}'=\ms{Z}\cup \ms{X}'$. In particular, $\ms{C}=\ms{Z}'\cup \ms{X}_{Y(\p)}$.  We have $Q\w\in\frac{1}{h}\mc{I}_{\ms{Z}'}\Omega^q$ so that $\mc{I}_{\ms{X}_{Y(\p)}} Q\w\subseteq \frac{1}{h}\mc{I}_{\ms{C}}\Omega^q$. Let $f\in\mc{I}_{\ms{X}_{Y(\p)}}$. Let us prove that $\dd f\wedge Q\w\in\frac{1}{h}\mc{I}_\ms{C}\Omega^{q+1}$. Since $fQ\in\mc{I}_{\ms{X}_Y}$ and $\w\in\Omega^q(\log\ms{X}_Y/\ms{C})$, we have $\dd(fQ)\wedge \w\in\frac{1}{h}\mc{I}_\ms{C}\Omega^q$. We thus have, since $\mc{I}_\ms{C}\subseteq\mc{I}_{\ms{X}_{Y(\p)}}$:
$$\dd f\wedge Q\w=\dd(fQ)\wedge \w-f\dd Q\wedge \w\in\frac{1}{h}\mc{I}_{\ms{X}_{Y(\p)}}\Omega^{q+1}.$$

Since in addition $Q\w\in\frac{1}{h}\mc{I}_{\ms{Z}'}\Omega^q$, and $\mc{I}_\ms{C}=\mc{I}_{\ms{Z}'}\cap \mc{I}_{\ms{X}_{Y(\p)}}$, we have $\dd f\wedge Q\w\in\frac{1}{h}\mc{I}_{\ms{C}}\Omega^{q+1}$. Therefore, $Q\w\in\Omega^q(\log (\ms{X}_{Y(\p)}/\ms{C}))$. Since $Q\notin \p$, it shows that we have $$ \lrp{\Omega^q(\log (\ms{X}_{Y(\p)}/\ms{C})}_\p=\lrp{\Omega^q(\log(\ms{X}_{Y}/\ms{C}))}_\p.$$
\end{proof}

We can now prove property~\eqref{char:chi:4}.

We have:
\begin{align*}
\sum_{\substack{Z\in L_Y}}\mu(Z,Y)G(Z)&=\sum_{{Z\in L_Y}}\mu(Z,Y)\Psi(\Omega^\bullet{\log (\ms{X}_Z/\ms{C})},1,t)\\
 &=\sum_{Z\in L_Y} \sum_{q=0}^\ell \mu(Z,Y)\Po(\Omega^q(\log (\ms{X}_Z/\ms{C}),x)(t(1-x)-1)^q\Big|_{x=1}\\
 &=\sum_{q=0}^\ell  \lrp{\sum_{Z\in L_Y} \mu(Z,Y)\Po(\Omega^q(\log (\ms{X}_Z/\ms{C}),x)}(t(1-x)-1)^q\Big|_{x=1}\\
 &=\sum_{q=0}^\ell M_q(x)(t(1-x)-1)^q
\end{align*}
where $M_q(x)=\sum_{Z\in L_Y} \mu(Z,Y)\Po(\Omega^q(\log (\ms{X}_Z/\ms{C}),x)=\sum_{Z\in L_Y} \mu(Z,Y)\Po(F_q(Z),x)$.

By propositions~\ref{prop:local-functor-mu} and~\ref{prop:loc:func:res}, $(1-x)^{\dim(Y)}M_q(x)$ has no pole at $x=1$. As in the proof of~\cite[Theorem 4.136 (4)]{orlik-terao-hyperplanes}, we deduce that for all $n\supe \dim(Y)$, the coefficient of $t^n$ in $M_q(x)(t(1-x)-1)^q$ lies in $(1-x)\Z[x,x^{-1}]$. Hence the result.

\section{Examples}
\label{sec:ex}
We give in this section several examples. We first show that any line arrangement satisfies the generalized Solomon-Terao formula. We then give an example of surface in $\C^4$ which does not satisfy the formula. 

\subsection{Case of line arrangements}
\label{line}

We already considered the case of complete intersection line arrangements in proposition~\ref{prop:hom:ic:curves:psi}. Let us prove that any line arrangement in $\C^\ell$ satisfies the condition of theorem~\ref{theo:subspace:solomon:terao}, so that the generalization of Solomon Terao formula holds for any line arrangement. 

\begin{cor}
\label{cor-line-psi}
For any line arrangement $\ms{X}$ in  $\C^\ell$, $$\chi(\ms{X},t)=t^\ell-\Psi(\mc{R}_{\ms{X}}^\bullet,1,t). $$
\end{cor}
\begin{proof}
If $\ms{X}$ is composed of only one line, then $\ms{X}$ is smooth and we have $\mc{R}_\ms{X}=\co_{\ms{X}}$ and $\mc{R}_{\ms{X}}^1=\co_{\ms{X}}$. Therefore, $$\Psi(\mc{R}_\ms{X}^\bullet,x,t)=\Po(\co_{\ms{X}},x,t)+(t(1-x)-1)\Po(\co_{\ms{X}},x,t)= t(1-x)\Po(\co_\ms{X},x,t).$$ In addition, $\Po(\co_\ms{X},x)=\frac{1}{1-x}$. Therefore, $t^\ell-\Psi(\mc{R}_\ms{X}^\bullet,1,t)=t^\ell-t=\chi(\ms{X},t)$. 

\medskip

From now on, we will assume that $\ms{X}$ has at least two components. In particular, it means that $\ms{X}$ is singular.

Thanks to theorem~\ref{theo:subspace:solomon:terao}, it is sufficient to prove that for any line arrangement $\ms{X}$ in $\C^\ell$ we have $\Psi(\mc{R}_{\ms{X}}^\bullet,1,1)=1$. We will in fact prove that $\Psi(\mc{R}_\ms{X}^\bullet,x,1)=1$.

Let us denote $\ms{C}=\ms{X}\cup \ms{Y}$ where $\ms{Y}$ is the union of the irreducible components of $\ms{C}$ which are not in $\ms{X}$.

By remark~\ref{remar:incl:res}, we have an inclusion $\mc{R}_\ms{X}\hookrightarrow \mc{R}_\ms{C}$. Thanks to this inclusion, we will consider $\mc{R}_\ms{X}$ as a submodule of $\mc{R}_\ms{C}$. More precisely, $\mc{R}_{\ms{X}}$ can be identified with $\lra{\rho\in\mc{R}_{\ms{C}}\ ;\ {\rho|}_{\ms{Y}}=0}$ (see \cite[Proposition 4.2.7]{polthese}).

Let $\mathrm{Jac}(h_1,\ldots,h_{\ell-1})$ be the Jacobian matrix associated with $(h_1,\ldots,h_{\ell-1})$. For $i\in \lra{1,\ldots,\ell}$ we denote by $J_i$ the minor $(\ell-1)\times (\ell-1)$ of $\mathrm{Jac}(h_1,\ldots,h_{\ell-1})$ obtained by removing the $i$th column. Let $c_1,\ldots,c_\ell\in\C$ be such that $g=c_1J_1+\ldots+c_\ell J_\ell\in S$ and $y=\sum_{i=1}^\ell (-1)^{i-1} c_ix_i\in S$  induce non zero divisors of $\co_\ms{C}$. 

By \cite[Proposition 6.1.24]{polthese}, the module $\mc{R}_\ms{C}$ is generated by $1=\frac{g}{g}$ and $\frac{y}{g}$. 

In addition, for all $i,j\in\lra{1,\ldots,\ell}$ we have (see \cite[(6.13)]{polthese}):
$$(-1)^{j-1} x_jJ_i=(-1)^{i-1}x_iJ_j  \mod \mc{I}_\ms{C}.$$

We then have for all $j\in \lra{1,\ldots,\ell}$:
\begin{equation}
\label{eq-y-g}
x_jg=(-1)^{j-1}J_jy \mod \mc{I}_\ms{C}. 
\end{equation}

Therefore, for all $\rho\in \mc{R}_\ms{C}$, there exists $a\in \co_\ms{C}$ and $c\in \C$ such that $\rho=a\frac{y}{g}+c$.

We first prove the following lemmas.

\begin{lem}
\label{lem-dh-rx}
There exists $a_1\in\co_\ms{C}$ and $c_1\in\C$, $c_1\neq 0$ such that\footnote{Since $\mc{R}_\ms{X}$ is graded, we may assume that $\rho_1$ is an homogeneous element.} $\rho_1=a_1\frac{y}{g}+c_1\in\mc{R}_\ms{X}\subseteq \mc{R}_\ms{C}$.
\end{lem}
\begin{proof}
Let us assume that for all $\rho=a\frac{y}{g}+c\in\mc{R}_\ms{X}$, we have $c=0$. Then it means that  $\mc{R}_\ms{X}$ is contained in the $\co_\ms{C}$-module generated by $\frac{y}{g}$.

Let us consider the value map on $\ms{C}$ as in \cite[Definition 2.1]{polvalues}. Let us denote by $p$ the number of irreducible components of $\ms{C}$. We still assume that $\ms{X}$ has at least two components, so that $p\geqslant 2$. 

\smallskip

The value map $\val$ associates with any element $f\in\mathrm{Frac}(\co_\ms{C})$ the $p$-uple of its valuation along each irreducible component of $\ms{C}$. We denote $\ms{C}=\ms{X}_1\cup\dots\cup \ms{X}_s\cup\ms{Y}_{s+1}\cup\dots\cup\ms{Y}_p$ where the $\ms{X}_i$ are the irreducible components of $\ms{X}$. For all $a,b\in\mathrm{Frac}(\co_{\ms{C}})$, we have $\val(ab)=\val(a)+\val(b)$.

\smallskip

If $I\subseteq \mathrm{Frac}(\co_\ms{C})$ is an ideal, we set $\val(I)=\lra{\val(g)\ ;\ g\in I}\cap \Z^p$. 

\smallskip

We denote by $\co_{\wt{\ms{C}}}\subseteq \mathrm{Frac}(\co_{\ms{C}})$ the normalization of $\co_{\ms{C}}$. In particular, since $\ms{C}$ is the union of $p$ lines, $\co_{\wt{\ms{C}}}\simeq \bigoplus_{i=1}^p \C[t_i].$ In addition, by \cite[Proposition 3.1.28]{polthese}, $\co_{\wt{\ms{C}}}\subseteq \mc{R}_\ms{C}$. We denote by $\gamma=(\gamma_1,\ldots,\gamma_p)\in\N^p$ the conductor of $\ms{C}$, which is the lowest element for the product order in $\Z^p$ which satisfies $\gamma+\N^p\subseteq \val(\co_{\ms{C}})$ (see \cite[Lemma  2.8]{polvalues}).

\smallskip

 Let $\theta\in\co_{\wt{\ms{C}}}\subseteq \mc{R}_\ms{C}$ be such that $\theta|_{\ms{X}_1}=1$ and for all $i\in\lra{2,\ldots,s}$, for all $j\in\lra{s+1,\ldots,p}$, $\theta|_{\ms{X}_i}=0$ and $\theta|_{\ms{Y}_j}=0$. Then $\theta\in\mc{R}_\ms{X}$ since $\theta|_\ms{Y}=0$. Thus, by assumption, there exists $a\in\co_\ms{C}$ such that $\theta=a\frac{y}{g}$. In particular, $\val\lrp{a\frac{y}{g}}=(0,\infty,\ldots,\infty)$. 
 
By \cite[Lemma 6.1.22]{polthese}, we have $\val\lrp{\frac{y}{g}}=-\gamma+\underline{1}$. Let $b\in\co_{\ms{C}}$ be such that $\val(b)=\gamma$.  Then $\val\lrp{b\frac{y}{g}}=(1,\ldots,1)$. Therefore, using \cite[Proposition 2.10]{polvalues}, we have $\val\lrp{(a+b)\frac{y}{g}}=(0,1,\ldots,1)$. Since $\val\lrp{(a+b)\frac{y}{g}}=\val(a+b)+\val\lrp{\frac{y}{g}}$, we have $\val(a+b)=(\gamma_1-1,\gamma_2,\ldots,\gamma_p)\in\val(\co_{\ms{C}})$. 

However,  the set $\lra{v\in\val(\co_\ms{C})\ ;\ v_1=\gamma_1-1 \text{ and } \forall j\neq 1, v_j\geqslant \gamma_j}$ is empty (see \cite[Proposition 2.17]{polvalues}). 

Hence the result.
\end{proof}

\begin{lem}
\label{lem-gen-rx-pascx}
The module $\mc{R}_\ms{X}$ is generated as an $\co_\ms{C}$-module and as an $\co_\ms{X}$-module by the elements in $\mc{I}_\ms{Y}\frac{y}{g}$ and the residue $\rho_1$ introduced in lemma~\ref{lem-dh-rx}.
\end{lem}

\begin{proof}
By \cite[Proposition 4.2.7]{polthese}, for all $\rho\in \mc{R}_\ms{C}$ we have $\rho\in \mc{R}_\ms{X}$ if and only if the restriction of $\rho$ to $\ms{Y}$ is zero. Therefore, we have $\mc{I}_\ms{Y}\frac{y}{g}\subseteq \mc{R}_\ms{X}$.

Let $\rho\in\mc{R}_\ms{X}$.  We denote by $\ov{\mc{I}_\ms{Y}}\subseteq \co_\ms{C}$ the image of $\mc{I}_\ms{Y}$ in $\co_\ms{C}$. There exists $a\in\co_{\ms{C}}$ and $c\in\C$ such that $\rho=a\frac{y}{g}+c$. We thus have: \begin{equation}
\label{eq-rho}
ay+cg\in \ov{\mc{I}_\ms{Y}}.
\end{equation}

We have:
\begin{align*}
ay+cg&=ay+\frac{c}{c_1}(a_1y+c_1g)-\frac{c}{c_1}a_1y\\
     &=(a-\frac{c}{c_1}a_1)y+\frac{c}{c_1}(a_1y+c_1g)\in\ov{\mc{I}_{\ms{Y}}}
\end{align*}
Since $\frac{a_1y+c_1g}{g}\in\mc{R}_\ms{X}$, we have $a_1y+c_1g\in\ov{\mc{I}_{\ms{Y}}}$ and therefore $(a-\frac{c}{c_1}a_1)y\in\ov{\mc{I}_\ms{Y}}$. Since $y$ induces a non zero divisor in $\co_\ms{C}$, we have $a-\frac{c}{c_1}a_1\in\ov{\mc{I}_\ms{Y}}$. Thus, since $\mc{I}_\ms{C}\subseteq \mc{I}_\ms{Y}$, we have:
$$\rho=\mu \frac{y}{g}+\frac{c}{c_1} \rho_1$$
with $\mu\in\mc{I}_\ms{Y}$. 
\end{proof} 

Let us prove corollary~\ref{cor-line-psi}. We have $\mc{R}_\ms{X}=\frac{1}{g}\lrp{\mc{I}_\ms{Y}\frac{y}{g}+\co_\ms{C}\frac{a_1y+c_1g}{g}}$. We have the following exact sequence:
\begin{equation}
\label{eq:rx:ic:M}
 0\to \ov{\mc{I}_\ms{Y}}\xrightarrow{y} y\ov{\mc{I}_\ms{Y}}+(a_1y+c_1g)\co_\ms{C}\to \frac{(a_1y+c_1g)S+y\mc{I}_\ms{Y}+\mc{I}_\ms{C}}{y\mc{I}_\ms{Y}+\mc{I}_\ms{C}}\to 0.
 \end{equation}

Let us compute a free resolution of the module $M=\frac{(a_1y+c_1g)S+y\mc{I}_\ms{Y}+\mc{I}_\ms{C}}{y\mc{I}_\ms{Y}+\mc{I}_\ms{C}}$. We have the following exact sequence:
$$0\to \big( (y\mc{I}_\ms{Y}+\mc{I}_\ms{C}):(a_1y+c_1g)\big)_S\to S\xrightarrow{(a_1y+c_1g)} M\to 0.$$ 

Let us compute $\mc{T}:=\big( (y\mc{I}_\ms{Y}+\mc{I}_\ms{C}):(a_1y+c_1g)\big)_S$. Let $i\in\lra{1,\ldots,\ell}$. Let us prove that $x_i\in\mc{T}$. We have by~\eqref{eq-y-g}:
\begin{align*}
x_i(a_1y+c_1g)&=x_ia_1y+c_1x_ig\\
              &=(x_ia_1+(-1)^{i-1}J_i)y+\lambda\in\mc{I}_{\ms{Y}}
\end{align*}
with $\lambda\in\mc{I}_\ms{C}$. Since $\mc{I}_\ms{C}\subseteq \mc{I}_\ms{Y}$ and $y$ is a non zero divisor in $\co_\ms{C}$, we have $(x_ia_1+(-1)^{i-1}J_i)\in\mc{I}_\ms{Y}$, so that $x_i(a_1y+c_1g)\in y\mc{I}_{\ms{Y}}+\mc{I}_\ms{C}$ and $x_i\in \mc{T}$. Therefore, $(x_1,\ldots,x_\ell)\subseteq \mc{T}$. 

In addition, from lemma~\ref{lem-dh-rx}, we have $(a_1y+c_1g)\notin y\ov{\mc{I}_\ms{Y}}$, so that we have $\lrb{x_1,\ldots,x_\ell}= \mc{T}$. Then a minimal free resolution of $\big( (y\mc{I}_\ms{Y}+\mc{I}_\ms{C}):(a_1y+c_1g)\big)_S$ is deduced from the Koszul complex associated with the regular sequence $(x_1,\ldots,x_\ell)$:
\begin{equation}
\label{eq:free:res:M} 
0\to S(-(d-\ell+1)-\ell)^{\binom{\ell}{\ell}}\to  \cdots\to  S(-(d-\ell+1)-1)^{\binom{\ell}{1}}\to S(-(d-\ell+1))^{\binom{\ell}{0}}\to M\to 0.
\end{equation}

From the exact sequence~\eqref{eq:rx:ic:M} and the additivity of Poincar\'e series, we have:
$$\Po(y\ov{\mc{I}_\ms{C}}+(a_1y+c_1g)\co_\ms{C},x)=x\Po(\ov{\mc{I}_{\ms{Y}}},x)+\Po(M,x).$$
Since $\mc{R}_\ms{X}=\frac{1}{g}\lrp{y\ov{\mc{I}_{\ms{Y}}}+(a_1y+c_1g)\co_\ms{C}}$, we have:
$$\Po(\mc{R}_\ms{X},x)=x^{\ell-d}\Po(\ov{\mc{I}_\ms{Y}},x)+x^{\ell-d-1}\Po(M,x).$$

From the free resolution~\eqref{eq:free:res:M}, we have $\Po(M,x)=x^{d-\ell+1}\sum_{i=0}^\ell \binom{\ell}{i} \frac{x^i}{(1-x)^\ell}=x^{d-\ell+1}$. Therefore, 
\begin{equation}
\label{eq:po:res:line}
\Po(\mc{R}_\ms{X},x)=x^{\ell-d}\Po(\ov{\mc{I}_\ms{Y}},x)+1.
\end{equation}

Let us consider $\mc{R}_\ms{X}^1$. We have $\Omega^\ell(\log\ms{C})=\frac{1}{h}\Omega^\ell$. From definition~\ref{de:formes:loga}, we deduce that $$\Omega^\ell(\log\ms{X}/\ms{C})=\frac{1}{h}\mc{I}_{\ms{Y}}\Omega^\ell.$$

Therefore, $\mc{R}_\ms{X}^1=\mc{I}_\ms{Y} \mathrm{res}_{\ms{C}}\lrp{\frac{\dd x_1\wedge \dots\wedge \dd x_\ell}{h}}\simeq \ov{\mc{I}_{\ms{Y}}}$. In addition, $\mathrm{res}_{\ms{C}}\lrp{\frac{\dd x_1\wedge\dots\wedge \dd x_\ell}{h}}$ is homogeneous of degree $-(d-\ell+1)$.

We thus have: 
$$\Po(\mc{R}_\ms{X}^1,x)=x^{-d+\ell-1}\Po(\ov{\mc{I}_\ms{Y}},x).$$

We can now compute the $\Psi$-function associated with the modules of logarithmic residues:
\begin{align*}
\Psi(\mc{R}_\ms{X}^\bullet,x,t)&=\Po(\mc{R}_\ms{X},x)+(t(1-x)-1)\Po(\mc{R}_\ms{X}^1,x)\\
          &=1+x^{\ell-d}\Po(\ov{\mc{I}_\ms{Y}},x)+(t(1-x)-1)x^{\ell-d-1}\Po(\ov{\mc{I}_\ms{Y}},x)\\
          &=1+x^{\ell-d-1}\Po(\ov{\mc{I}_\ms{Y}},x)(x+t(1-x)-1)\\
          &=1+x^{\ell-d-1}\Po(\ov{\mc{I}_\ms{Y}},x)(t-1)(1-x)
\end{align*}

Therefore, $\Psi(\mc{R}_\ms{X}^\bullet,x,1)=1$, which gives us corollary~\ref{cor-line-psi}.
\end{proof}

\begin{remar}
The computation of $\mc{R}_\ms{X}$ made in the proof of corollary~\ref{cor-line-psi} is not specific to line arrangements. Let $C\subseteq (\C^\ell,0)$ be a reduced singular complete intersection curve which is quasi-homogeneous with respect to the weights $(w_1,\ldots,w_\ell)$. Let $c_1,\ldots,c_\ell\in\C$ be such that $g=\sum_{i=1}^\ell c_iJ_i$ and $y=\sum_{i=1}^\ell (-1)^{i-1}c_iw_ix_i$ induce non zero divisors in $\co_C$.  By \cite[Proposition 6.1.24]{polthese}, the module $\mc{R}_C$ is generated by $1$ and $\frac{y}{g}$. Then with exactly the same proof as for corollary~\ref{cor-line-psi}, one can prove that for any equidimensional reduced singular subspace $X\subseteq C$ of dimension $1$, there exists $\rho_1=a_1\frac{y}{g}+c_1\in\mc{R}_X$ with $c_1\neq0$ and that the module of logarithmic multi-residues $\mc{R}_X$ is generated by the elements in $\mc{I}_Y\frac{y}{g}$ and $\rho_1$, where $Y$ denotes the union of the components of $C$ which are not in $X$. 
\end{remar}

\subsection{Higher dimensional subspace arrangements}
\label{dim:sup}

The question which then arises is to determine if the condition $\Psi(\mc{R}_\ms{X}^\bullet,1,1)=1$ is satisfied for any equidimensional subspace arrangement of any dimension. The answer is no, as it is shown by the following example.

\begin{ex}
\label{ex:xy:zt}
Let us consider $\mc{I}_\ms{X}=\lrb{x,z}\cap\lrb{y,t}=\lrb{xy,xt,yz,zt}$. A reduced complete intersection subspace arrangement containing $\ms{X}$ is given by the ideal $\mc{I}_\ms{C}=\lrb{xy,zt}$. We set $h_1=xy$, $h_2=zt$, $h=xyzt$. Let $\mc{I}_\ms{Y}=\lrb{xy,xz,yt,zt}$ be the radical ideal defining the union of the irreducible components of $\ms{C}$ which are not contained in $\ms{X}$. 

Let us compute the module of multi-logarithmic differential forms. We have:
$$\Omega^0(\log \ms{X}/\ms{C})=\frac{1}{h}\mc{I}_\ms{C},\ \ \ \Omega^1(\log \ms{X}/\ms{C})=\frac{1}{h}\mc{I}_\ms{C}\Omega^1,\ \ \ \Omega^4(\log \ms{X}/\ms{C})=\frac{1}{h}\mc{I}_\ms{Y}\Omega^4.$$

Computations made with \textsc{Singular} gives that $$\frac{1}{h}\lrp{a_1\dd z\wedge \dd t+a_2\dd y\wedge\dd t+a_3\dd y\wedge\dd z+a_4\dd x\wedge\dd t+a_5\dd x\wedge \dd z+a_6\dd x\wedge \dd y}\in\Omega^2(\log \ms{X}/\ms{C})$$ if and only if $$a_1,a_3,a_4,a_6\in\mc{I}_\ms{C}, a_2\in\lrb{xz,xy,zt}, a_5\in \lrb{yt,xy,zt}.$$

In addition, we have:
$$\frac{1}{h}\lrp{a_1\dd y\wedge \dd z\wedge \dd t+a_2\dd x\wedge\dd z\wedge \dd t+a_3\dd x\wedge\dd y\wedge \dd t+a_4\dd x\wedge\dd y\wedge\dd z}\in\Omega^3(\log\ms{X}/\ms{C})$$ if and only if $$a_1,a_3\in\lrb{xz,xy,zt}, a_2,a_4\in\lrb{yt,xz,zt}.$$

Free resolutions of these modules are given by:
\begin{align*}
0\leftarrow \Omega^0(\log \ms{X}/\ms{C})&\leftarrow S(2)^2\leftarrow S(0)\leftarrow 0 \\
0\leftarrow \Omega^1(\log \ms{X}/\ms{C})&\leftarrow S(2)^4\leftarrow S(0)^2\leftarrow 0\\
0\leftarrow \Omega^2(\log \ms{X}/\ms{C})&\leftarrow S(2)^{14}\leftarrow S(0)^4\oplus S(1)^4\leftarrow 0\\
0\leftarrow \Omega^3(\log \ms{X}/\ms{C})&\leftarrow S(2)^{12}\leftarrow S(1)^8\leftarrow 0\\
0\leftarrow \Omega^4(\log \ms{X}/\ms{C})&\leftarrow S(2)^4\leftarrow S(1)^4\leftarrow S(0)\leftarrow 0
\end{align*}
It is then possible to compute that $$\Psi(\Omega^\bullet(\log \ms{X}/\ms{C}),x,t)=\frac{1}{x^4}\lrp{x^4t^4-4x^3t^4+4x^3t^3+4x^2t^4-4x^2t^3+2x^2t^2}.$$
In particular, $\Psi(\Omega^\bullet(\log \ms{X}/\ms{C}),1,t)=t^4+2t^2$ and $t^4-\Psi(\mc{R}_\ms{X}^\bullet,1,t)=t^4-2t^2$ whereas the characteristic polynomial is $\chi(\ms{X},t)=t^4-2t^2+1$. In particular, $\Psi(\mc{R}_\ms{X}^\bullet,1,1)=2$.
\end{ex}

However, there exist subspace arrangements of dimension greater than one such that the relation $t^\ell-\Psi(\mc{R}_\ms{X}^\bullet,1,t)=\chi(\ms{X},t)$ is satisfied, for example the complete intersection of $\C^4$ defined by $\lrb{xy,zt}$. 
An interesting problem would be to characterize the subspace arrangements for which  the generalized Solomon-Terao formula holds. 

\bibliographystyle{alpha}
\bibliography{bibli2}
\end{document}